\documentclass{ejmpreprint}

\usepackage{amsmath}
\usepackage{exscale}
\usepackage{amssymb}

\font\ppppcarac=ptmr8y at 5pt
\font\pppcarac=ptmr8y at 6pt
\font\ppcarac=ptmr8y at 7pt
\font\Pcarac=ptmr8y at 8pt

\newcommand{\bfA}{{\mathbf{A}}}

\newcommand{\bfC}{{\mathbf{C}}}

\newcommand{\bfG}{{\mathbf{G}}}

\newcommand{\bfK}{{\mathbf{K}}}
\newcommand{\bfL}{{\mathbf{L}}}

\newcommand{\bfU}{{\mathbf{U}}}

\newcommand{\bfW}{{\mathbf{W}}}

\newcommand{\bfY}{{\mathbf{Y}}}
\newcommand{\bfZ}{{\mathbf{Z}}}

\newcommand{\bfh}{{\mathbf{h}}}

\newcommand{\bfu}{{\mathbf{u}}}

\newcommand{\bfw}{{\mathbf{w}}}
\newcommand{\bfx}{{\mathbf{x}}}
\newcommand{\bfy}{{\mathbf{y}}}
\newcommand{\bfz}{{\mathbf{z}}}

\newcommand{\bfzero}{{\mathbf{0}}}

\newcommand{\bfPhi}{{\boldsymbol{\Phi}}}
\newcommand{\bfPsi}{{\boldsymbol{\Psi}}}
\newcommand{\bfXi}{{\boldsymbol{\Xi}}}

\newcommand{\bfalpha}{{\boldsymbol{\alpha}}}
\newcommand{\bfbeta}{{\boldsymbol{\beta}}}

\newcommand{\bfeta}{{\boldsymbol{\eta}}}

\def\LL{{\mathbb{L}}}
\def\MM{{\mathbb{M}}}
\def\NN{{\mathbb{N}}}

\def\OO{{\mathbb{O}}}
\def\PP{{\mathbb{P}}}

\def\RR{{\mathbb{R}}}

\def\VV{{\mathbb{V}}}

\DeclareMathAlphabet{\mathonebb}{U}{bbold}{m}{n}
\def\11{{\ensuremath{\mathonebb{1}}}}

\def\tt{{\ensuremath{\mathonebb{t}}}}

\newcommand{\curB}{{\mathcal{B}}}
\newcommand{\curC}{{\mathcal{C}}}
\newcommand{\curD}{{\mathcal{D}}}

\newcommand{\curG}{{\mathcal{G}}}

\newcommand{\curK}{{\mathcal{K}}}
\newcommand{\curL}{{\mathcal{L}}}
\newcommand{\curM}{{\mathcal{M}}}

\newcommand{\curP}{{\mathcal{P}}}

\newcommand{\curS}{{\mathcal{S}}}
\newcommand{\curT}{{\mathcal{T}}}

\newcommand{\Cdroit}{{\mathrm{C}}}

\newcommand{\SFGP}{{\hbox{{SFG}}^+}}

\newcommand{\trmat}{\mathrm{tr}}
\newcommand{\Dbar}{{\overline{D}}}
\newcommand{\obs}{{\hbox{{\ppcarac obs}}}}
\newcommand{\pobs}{{\hbox{{\pppcarac obs}}}}
\newcommand{\exper}{{\hbox{{\ppcarac exp}}}}
\newcommand{\experp}{{\hbox{{\pppcarac exp}}}}

\newcommand{\APM}{{\hbox{{\pppcarac APM}}}}

\newcommand{\OAPM}{{\hbox{{\pppcarac OAPM}}}}
\newcommand{\pOAPM}{{\hbox{{\ppppcarac OAPM}}}}

\newcommand{\opt}{{\hbox{{\ppcarac opt}}}}

\newcommand{\KL}{{\hbox{{\pppcarac KL}}}}
\newcommand{\prior}{{\hbox{{\Pcarac prior}}}}
\newcommand{\post}{{\hbox{{\Pcarac post}}}}

\let\realverbatim=\verbatim
\let\realendverbatim=\endverbatim
\renewcommand\verbatim{\par\addvspace{6pt plus 2pt minus 1pt}\realverbatim}
\renewcommand\endverbatim{\realendverbatim\addvspace{6pt plus 2pt minus 1pt}}

\ifprodtf \else
  \checkfont{eurm10}
  \iffontfound
    \IfFileExists{upmath.sty}
      {\typeout{^^JFound AMS Euler Roman fonts on the system,
                   using the 'upmath' package.^^J}%
       \usepackage{upmath}}
      {\typeout{^^JFound AMS Euler Roman fonts on the system, but you
                   don't seem to have the}%
       \typeout{'upmath' package installed. EJM.cls can take advantage
                 of these fonts,^^Jif you use 'upmath' package.^^J}%
      }
  \else
  \fi
\fi

\ifprodtf \else
  \checkfont{msam10}
  \iffontfound
    \IfFileExists{amssymb.sty}
      {\typeout{^^JFound AMS Symbol fonts on the system, using the
                'amssymb' package.^^J}%
       \usepackage{amssymb}%
       \let\le=\leqslant  \let\leq=\leqslant
       \let\ge=\geqslant  \let\geq=\geqslant
      }{}
  \fi
\fi

\ifprodtf \else
  \IfFileExists{amsbsy.sty}
    {\typeout{^^JFound the 'amsbsy' package on the system, using it.^^J}%
     \usepackage{amsbsy}}
    {\providecommand\boldsymbol[1]{\mbox{\boldmath $##1$}}}
\fi

\newsavebox{\astrutbox}
\sbox{\astrutbox}{\rule[-5pt]{0pt}{20pt}}

\newdefinition{definition}[theorem]{Definition}
\newtheorem{lemma}{Lemma}
\newtheorem{proposition}{Proposition}

\newdefinition{rem}[theorem]{Remark}

\title[Random fields representations for statistical inverse boundary value problems]{Random fields representations for stochastic elliptic boundary value problems and statistical inverse problems}

\author[A. NOUY AND C. SOIZE]{%
  A. NOUY$\,^1$ \and C. SOIZE$\,^2$
}

\affiliation{%
  $^1\,$Ecole Centrale de Nantes, LUNAM Universit\'e, GeM UMR CNRS 6183,  1 rue de la Noe, 44321 Nantes, France\\
    email\textup{\nocorr: \texttt{anthony.nouy@ec-nantes.fr}}\\
  $^2\,$Universit\'e Paris-Est, Laboratoire Mod\'elisation et Simulation Multi-Echelle,
MSME UMR 8208 CNRS, 5 bd Descartes, 77454 Marne-la-Vallée, France\\   email\textup{\nocorr: \texttt{christian.soize@univ-paris-est.fr}}}

\date{10 April 2013}
\pubyear{2000}
\volume{000}
\pagerange{\pageref{firstpage}--\pageref{lastpage}}

\begin{document}

\label{firstpage}
\maketitle

\thispagestyle{empty}

\begin{abstract} \textbf{Abstract}. 
This paper presents new results allowing an unknown non-Gaussian positive matrix-valued random field to be identified  through a stochastic elliptic boundary value problem, solving a statistical inverse problem. A new general class of non-Gaussian positive-definite matrix-valued random fields, adapted to the statistical inverse problems in high stochastic dimension for their experimental identification, is introduced and its properties are analyzed. A minimal parametrization of discretized random fields belonging to this general class is proposed.  Using this parametrization of the general class,  a complete identification procedure is proposed. New results of the mathematical and numerical analyzes of the parameterized stochastic elliptic boundary value problem 
are presented. The numerical solution of this parametric stochastic problem provides an explicit approximation of the application that maps the parameterized general class of random fields to the corresponding set of random solutions. This approximation can be used during the identification procedure in order to avoid the solution of multiple forward stochastic problems. Since the proposed general class of random fields possibly contains random fields which are not uniformly bounded, a particular mathematical analysis is developed and  dedicated approximation methods are introduced. In order to obtain an algorithm for constructing the approximation of a very high-dimensional map, complexity reduction methods are introduced and are based on the use of sparse or low-rank approximation methods that exploit the tensor structure of the solution which results from the parametrization of the general class of random fields.
\end{abstract}


\subsection*{Notations}
\

\noindent A lower case letter, $y$, is a  real deterministic variable.\\
A boldface lower case letter, $\bfy=(y_1,\ldots, y_N)$ is a real deterministic vector.\\
An upper case letter, $Y$, is a real random variable.\\
A boldface upper case letter, $\bfY=(Y_1,\ldots, Y_N)$ is a real random vector.\\
A lower (or an upper) case letter between brackets, $[a]$ (or $[A]$), is a real deterministic matrix.\\
A boldface upper case letter between brackets, $[\bfA]$, is a real random matrix.\\

\noindent
$E$: Mathematical expectation.\\
$\MM_{N,m}(\RR)$: set of all the $(N\times m)$ real matrices.\\
$\MM_{m}(\RR)$: set of all the square $(m\times m)$ real matrices.\\
$\MM_n^{\rm S}(\RR)$: set of all the symmetric $(n\times n)$ real matrices.\\
$\MM_n^+(\RR)$: set of all the symmetric positive-definite $(n\times n)$ real matrices.\\
$\MM_m^{\rm SS}(\RR)$: set of all the skew-symmetric $(m\times m)$ real matrices.\\
$\MM_n^{\rm U}(\RR)$:  set of all the upper triangular $(n\times n)$ real matrices with strictly positive diagonal.\\
$\OO(N)$: set of all the orthogonal $(N\times N)$ real matrices.\\
$\VV_m(\RR^N)$: compact Stiefel manifold of $(N\times m)$ orthogonal real matrices.\\
$\trmat$: Trace of a matrix.\\
$< \bfy  \, , \bfz >_2$: Euclidean inner product of $\bfy$ with $\bfz$ in $\RR^n$.\\
$\Vert \bfy \Vert_2$: Euclidean norm of a vector $\bfy$ in $\RR^n$.\\
$\Vert A \Vert_2$: Operator norm subordinated to $\RR^n$: $\Vert A \Vert_2 = \sup_{\Vert \bfy \Vert_2=1}\Vert A\bfy \Vert_2$.\\
$\Vert A \Vert_F$: Frobenius norm of a matrix such that $\Vert A \Vert_F^2=\trmat\{[A]^T\, [A]\}$.\\
$[I_{m,n}]$: matrix in $\MM_{m,n}(\RR)$ such that $[I_{m,n}]_{ij} = \delta_{ij}$.\\
$[I_{m}]$: identity matrix in $\MM_{m}(\RR)$.

\section{Introduction}
\label{Section1}
The experimental identification of an unknown non-Gaussian positive matrix-valued random field, using partial and limited experimental data for an observation vector related to the random solution of a stochastic elliptic boundary value problem, stays a challenging problem which has not received yet a complete solution, in particular for high stochastic dimension. An example of a challenging application concerns the identification of random fields that model at a mesoscale the elastic properties of complex heterogeneous materials that can not be described at the level of their microscopic constituents (e.g. biological materials such as the cortical bone), given some indirect observations measured on a collection of materials samples (measurements of their elastic response e.g. through image analysis or other displacement sensors).
Even if many results have already been obtained, additional developments concerning the stochastic representations of such random fields and additional mathematical and numerical analyzes of the associated stochastic elliptic boundary value problem must yet be produced. This is the objective of the present paper.\par

Concerning the representation of random fields adapted to the statistical inverse problems for their experimental identification, one interesting type of representation for any non-Gaussian second-order random field is based on the use of the polynomial chaos expansion \cite{Wiener1938,Cameron1947} for which an efficient construction has been proposed in \cite{Ghanem1991,Ghanem1996,Doostan2007}, consisting in coupling a Karhunen-Lo\`eve expansion (allowing a statistical reduction to be done) with a polynomial chaos expansion of the reduced model. This type of construction has been extended for an arbitrary probability measure \cite{Xiu2002,LeMaitre2004,Lucor2004,Soize2004,Wan2006,Ernst2012} and for random coefficients of the expansion \cite{Soize2009}.\par

Concerning the identification of random fields by solving stochastic inverse problems, works can be found such as \cite{DeOliveira1997,Hristopulos2003,Kaipio2005,Stuart2010} and some methods and formulations have been proposed for the experimental identification of non-Gaussian random fields  \cite{Desceliers2006,Ghanem2006,Das2008,Zabaras2008,Das2009,Marzouk2009,Stefanou2009,Arnst2010} in low stochastic dimension. More recently, a more advanced methodology \cite{Soize2010,Soize2011} has been proposed for the identification of non-Gaussian positive-definite matrix-valued random fields in high stochastic dimension for the case for which only partial and limited experimental data are available.\par

Concerning the mathematical analysis of stochastic elliptic boundary value problems and the associated numerical aspects for approximating  the random solutions (the forward problem), many works have been devoted to the simple case of uniformly elliptic and uniformly bounded operators \cite{Deb2001,Babuska2004,Babuska2005,Matthies2005,Frauenfelder2005, Babuska2007,Wan2009,Nouy2009,LeMaitre2010}. For these problems, existence and uniqueness results can be directly obtained using Lax-Milgram theorem and Galerkin approximation methods in classical approximation spaces can be used.  More recently, some works have addressed the mathematical and numerical analyzes of some classes of non uniformly elliptic operators \cite{Galvis2009,Gittelson2010, Mugler2011,Charrier2012}, including the case of log-normal random fields.

For the numerical solution of stochastic boundary value problems, classical non adapted choices of stochastic approximation spaces lead to very high-dimensional representations of the random solutions, e.g. when using classical (possibly piecewise)  polynomial spaces in the random variables. For addressing this complexity issue, several complexity reduction methods have been recently proposed, such as model reduction methods for stochastic and parametric partial differential equations \cite{Nouy2007,Nouy2008,Boyaval2009}, sparse approximation
methods for high-dimensional approximations \cite{Todor2007,Nobile2008,Cohen2010,Babuska2010,Schwab2011}, or low-rank tensor approximation methods that exploit the tensor structure of stochastic approximation spaces and also allow the representation of high-dimensional functions \cite{Doostan2009,Khoromskij2011,Nouy2010,Cances2011,Falco2011,Falco2012,Matthies2012}. \\
\par

In this paper, we present new results concerning:
\begin{enumerate}[(iii)]\renewcommand{\theenumi}{\roman{enumi}}
\item  the construction of a general class of non-Gaussian positive-definite matrix-valued random fields which is adapted to the identification in high stochastic dimension, presented in Section \ref{Section2},
\item the parametrization of the discretized random fields belonging to the general class and an adapted identification strategy for this class, presented in Section \ref{Section3},
\item the mathematical analysis of the parameterized stochastic elliptic boundary value problems whose random coefficients belong to the parameterized general class of random fields, and the introduction and analysis of dedicated approximation methods, presented in Section \ref{Section4}.
\end{enumerate}

{Concerning (i), we introduce a new class of matrix-valued random fields $\{[\bfK(\bfx)], \bfx\in D\}$ indexed on a domain $D$ that are expressed as a nonlinear function of second order symmetric matrix-valued random fields $ \{[\bfG(\bfx)], \bfx\in D\}$. More precisely, we propose and analyze two classes of random fields that use two different nonlinear mappings. The first class, which uses a matrix exponential mapping, contains the class of (shifted) lognormal random fields (when $[\bfG]$ is a Gaussian random field) and can be seen as a generalization of this classical class. The second class contains  (a shifted version of) the class $SFE^+$ of positive-definite matrix-valued random fields introduced in \cite{Soize2006} which arises from a maximum entropy principle given natural constraints on the integrability of random matrices and their inverse.}

{Concerning (ii), and following \cite{Soize2010}, we introduce a parametrization of the second order random fields $[\bfG]$ using Karhunen-Loeve and Polynomial Chaos expansions, therefore yielding a parametrization of the class of random fields in terms of the set of coefficients of the Polynomial Chaos expansions of the random variables of the Karhunen-Loeve expansions. Second order statistical properties of $[\bfG]$ are imposed by enforcing orthogonality constraints on the set of coefficients, therefore yielding a class of random fields parametrized on the compact Stiefel manifold. It finally results in a parametrized class of random fields $\mathcal{K}=\{[\bfK(\cdot)] = \mathcal{F}(\cdot,\bfXi,\bfz) ; \bfz \in \mathbb{R}^{\nu}\}$, where the parameters $\bfz$ are associated with a parametrization of the compact Stiefel manifold, and where $\bfXi$ is the Gaussian germ of the Polynomial Chaos expansion. 
A general procedure is then proposed for the identification of random fields in this new class, this procedure being an adaptation for the present parametrization of the procedure described in \cite{Soize2011}. The parametrization of the Stiefel manifold, some additional mathematical properties of the resulting parametrized class of random fields and the general identification procedure are introduced in Section \ref{Section3}.}

The numerical solution of the parameterized stochastic boundary value problems provides an explicit approximation of the application $u : D\times \mathbb{R}^{N_g}\times \mathbb{R}^\nu \rightarrow \mathbb{R}$ that maps the parameterized general class of random fields to the corresponding set of random solutions $\{u(\cdot,\bfXi,\bfz) ; \bfz \in \mathbb{R}^{\nu} \}$. This explicit map can be efficiently used in the identification procedure in order to avoid the solution of multiple stochastic forward problems. 
{Concerning (iii)}, the general class of random fields possibly contain random fields which are not uniformly bounded, which requires a particular mathematical analysis and the introduction of dedicated approximation methods. {In particular, in Section \ref{Section4}, we prove the existence and uniqueness of a weak solution $u$ in the natural function space associated with the energy norm, and we propose suitable constructions of approximation spaces for Galerkin approximations.}
Also, since the solution of this problem requires the approximation of a very high-dimensional map, {complexity reduction methods are required. The solution map has a tensor structure which is inherited from 
the particular parametrization of the general class of random fields. This allows the use of complexity reduction methods based on low-rank or  sparse  tensor approximation methods. The applicability of these reduction techniques in the present context is briefly discussed in Section \ref{Section4}.3. }
Note that this kind of approach for stochastic inverse problems has been recently analyzed in \cite{Schwab2012} in a particular Bayesian setting with another type of representation of random fields, and with sparse approximation methods for the approximation of the high-dimensional functions. 

\section{General class of non-Gaussian positive-definite matrix-valued random fields}
\label{Section2}
\subsection{Elements for the construction of a general class of random fields}
\label{Section2.1}

Let $d\geq 1$ and $n\geq 1$ be two integers. Let $D$ be a bounded open domain of $\RR^d$, let $\partial D$ be its boundary and  let $\Dbar = D\cup \partial D$ be its closure.\par

Let $\curC^+_n$ be the set of all the random fields $[\bfK] = \{[\bfK(\bfx)], \bfx\in D\}$, defined on a probability space $(\Theta,\curT,\curP)$, with values in $\MM_n^+(\RR)$.
A random field  $[\bfK]$ in $\curC^+_n$ corresponds to the coefficients of a stochastic elliptic operator and must be identified from data, by solving a statistical inverse boundary value problem. The objective of this section is to construct a general representation of $[\bfK]$ which allows its identification to be performed using experimental data and solving a statistical inverse problem based on the explicit construction of the solution of the forward (direct) stochastic boundary value problem. For that, we need to introduce some hypotheses for random field $[\bfK]$ which requires the introduction of a subset of $\curC^+_n$.

\par 
In order to normalize random field $[\bfK]$, we introduce a function $\bfx\mapsto [\underline{K}(\bfx)]$ from $D$ into $\MM_n^+(\RR)$ such that, for all $\bfx$ in $D$ and for all $\bfz$ in $\RR^n$,
\begin{equation}                                                                                                                               \label{Eq1}
\underline k_{\, 0} \Vert \bfz\Vert^2_2 \,\, \leq \,\,\, < \![\underline{K}(\bfx) ]\, \bfz \, , \bfz >_2
\,\, \leq \,\,\, n^{-1/2} \widetilde{\underline k}_{\, 1}  \Vert \bfz\Vert^2_2   \, ,
\end{equation}
in which $\underline k_{\, 0}$ and $\widetilde{\underline k}_{\, 1}$ are positive real constants, independent of $\bfx$, such that $0 < \underline k_{\, 0} < \widetilde{\underline k}_{\, 1} < +\infty$. The right inequality means that $\Vert \underline{K}(\bfx) \Vert_2 \leq n^{-1/2} \widetilde{\underline k}_{\, 1}$ and using
$\Vert \underline{K}(\bfx) \Vert_F \leq \sqrt{n}\, \Vert\underline{K}(\bfx)\Vert_2$ yields
$\Vert\underline{K}(\bfx)\Vert_F \leq \widetilde{\underline k}_{\, 1}$. Since $0\leq [\underline{K}(\bfx)]_{jj}$ and
$[\underline{K}(\bfx)]_{jj}^2\leq \Vert\underline{K}(\bfx)\Vert_F^2$, it can be deduced that
\begin{equation}                                                                                                                               \label{Eq1bis}
\trmat [\underline{K}(\bfx)] \,\,\leq \,\,\underline k_{\, 1} \, ,
\end{equation}
in which $\underline k_{\, 1}$ is a positive real constant, independent of $\bfx$, such that $\underline k_{\, 1} = n \, \widetilde{\underline k}_{\, 1}$.\\

\noindent We then introduce the following representation of the lower-bounded random field $[\bfK]$ in $\curC^+_n$ such that, for all $\bfx$ in $D$,
\begin{equation}                                                                                                                               \label{Eq2}
[\bfK(\bfx)] =\frac{1}{1+\varepsilon} [\underline L(\bfx)]^T\,\{ \varepsilon [\, I_n] + [\bfK_0(\bfx)]\}  \, [\underline L(\bfx)] \, ,
\end{equation}
in which $\varepsilon >0$ is any fixed positive real number, where $[\, I_n]$ is the $(n\times n)$ identity matrix and where $[\underline L(\bfx)]$ is the upper triangular $(n\times n)$ real matrix such that, for all $\bfx$ in $D$,
$[\underline{K}(\bfx) ] = [\underline L(\bfx)]^T\,[\underline L(\bfx)]$ and where $[\bfK_0] = \{[\bfK_0(\bfx)], \bfx\in D\}$ is any random field in $\curC^+_n$.\\

\noindent  For instance, if function $[\underline{K}]$ is chosen as the mean function of random field $[\bfK]$, that is to say, if for all $\bfx$ in $D$, $[\underline{K}(\bfx)] = E\{[\bfK(\bfx)] \}$, then Eq.~(\ref{Eq2}) shows that $E\{[\bfK_0(\bfx)] \}$ must be equal to $[\, I_n]$, what shows that random field $[\bfK_0]$ is normalized.

\begin{lemma}\label{lemma1} 
If  $[\bfK_0]$ is any random field in $\curC^+_n$, then the random field $[\bfK]$, defined by Eq.~(\ref{Eq2}), is such that
\begin{enumerate}[(iii)]\renewcommand{\theenumi}{\roman{enumi}}

\item  for all $\bfx$ in $D$,
\begin{equation}                                                                                                                              \label{Eq3}
\Vert\bfK(\bfx)\Vert_F \leq \frac{\underline k_{\, 1}}{1+\varepsilon} (\sqrt{n}\, \varepsilon + \Vert\bfK_0(\bfx)\Vert_F)\quad a.s \, .
\end{equation}
\item for all $\bfz$ in $\RR^n$ and for all $\bfx$ in $D$,
\begin{equation}                                                                                                                               \label{Eq4}
\underline k_{\, \varepsilon} \Vert \bfz\Vert^2_2 \,\, \leq \,\,\, < \![\bfK(\bfx) ]\, \bfz \, , \bfz >_2 \quad a.s \, ,
\end{equation}
in which $\underline k_{\, \varepsilon} = \underline k_{\, 0} \, \varepsilon /(1+\varepsilon)$ is a positive constant independent of $\bfx$.
\item  for all $\bfx$ in $D$,
\begin{equation}                                                                                                                              \label{Eq5}
\Vert [\bfK(\bfx)]^{-1}\Vert_F \leq \frac{\sqrt{n}\,(1+\varepsilon)}{\varepsilon}\, \trmat [\underline K(\bfx)]^{-1}  \quad a.s \, ,
\end{equation}
which shows that, for all integer $p\geq 1$,  $\{ [\bfK(\bfx)]^{-1}, \bfx\in D\}$ is a $p$-order random field, \textit{i.e.}, for all $\bfx$ in $D$, $E\{ \Vert [\bfK(\bfx)]^{-1}\Vert_F^p\} < +\infty$ and in particular, is a second-order random field.
\end{enumerate}
\end{lemma}
%
%
\begin{proof}
(i) Taking the Frobenius norm of the two members of Eq.~(\ref{Eq2}), using $\Vert\underline L(\bfx)\Vert_F$ $= \Vert[\underline L(\bfx)]^T\Vert_F =\sqrt{\trmat[\underline K(\bfx)]}$ and taking into account Eq.~(\ref{Eq1bis}), we obtain Eq.~(\ref{Eq3}).\\
(ii) Eq.~(\ref{Eq4}) can easily be proven using Eq.~(\ref{Eq2}) and the left inequality in Eq.~(\ref{Eq1}).\\
(iii) We have $\Vert[\bfK(\bfx)]^{-1}\Vert_F \leq \Vert [\underline L(\bfx)]^{-1}\Vert_F \Vert [\underline L(\bfx)]^{-T}\Vert_F \Vert[\bfK_\varepsilon(\bfx)]^{-1}\Vert_F$ in which
$[\bfK_\varepsilon(\bfx)]= (\varepsilon [\, I_n] + [\bfK_0(\bfx)])/(1+\varepsilon)$. Since $[\bfK_0(\bfx)]$ is positive definite almost surely, for $\bfx$ fixed in $D$, we can write $[\bfK_0(\bfx)] = [\bfPhi(\bfx)]\, [\Lambda(\bfx)]\, [\bfPhi(\bfx)]^T$ in which $[\Lambda(\bfx)]$ is the diagonal random matrix of the positive-valued random  eigenvalues $\Lambda_1(\bfx),$ $\ldots ,$ $\Lambda_n(\bfx)$ of $[\bfK_0(\bfx)]$ and $[\bfPhi(\bfx)]$ is the orthogonal real random matrix  made up of the associated random eigenvectors. It can then be deduced that
$[\bfK_\varepsilon(\bfx)]^{-1} =  (1+\varepsilon)\, [\bfPhi(\bfx)]\, (\varepsilon [I_n] + [\Lambda(\bfx)])^{-1}\, [\bfPhi(\bfx)]^T$
and consequently, for $\bfx$  in $D$,
$\Vert[\bfK_\varepsilon(\bfx)]^{-1}\Vert^2_F$ $=\trmat\{[\bfK_\varepsilon(\bfx)]^{-2}\} = (1+\varepsilon)^2 \sum_{j=1}^n (\varepsilon +\Lambda_j(\bfx) )^{-2}\leq (1+\varepsilon)^2 n/\varepsilon^2$. Since
$\Vert [\underline L(\bfx)]^{-1}\Vert_F$ $\Vert [\underline L(\bfx)]^{-T}\Vert_F $ $ =\trmat[\underline K(\bfx)]^{-1} $,
we deduce Eq.~(\ref{Eq5}).
\end{proof}
%
%

\begin{lemma}\label{lemma2} 
If  $[\bfK_0]$ is any random field in $\curC^+_n$, such that, for all $\bfx$ in $D$,
\begin{enumerate}[(ii)]\renewcommand{\theenumi}{\roman{enumi}}

\item $\Vert\bfK_0(\bfx)\Vert_F \leq \beta_0 < \infty$ almost surely, in which $\beta_0$ is a positive-valued random variable independent of $\bfx$, then the random field $[\bfK]$, defined by Eq.~(\ref{Eq2}), is such that, for all $\bfx$ in $ D$,
\begin{equation}                                                                                                                               \label{Eq6}
\Vert\bfK(\bfx)\Vert_F \,\,\leq \,\, \beta < +\infty \quad a.s\, ,
\end{equation}
in which  $\beta$ is the positive-valued random variable independent of $\bfx$ such that $\beta =\underline k_{\, 1} (\sqrt{n}\, \varepsilon +\beta_0)/(1+\varepsilon)$.
\item  $E\{ \Vert\bfK_0(\bfx)\Vert^2_F \} < + \infty$, then $[\bfK]$ is a second-order random field,
\begin{equation}                                                                                                                               \label{Eq7}
E\{\Vert\bfK(\bfx)\Vert^2_F\} < +\infty \, .
\end{equation}
\end{enumerate}
\end{lemma}
%
%
\begin{proof}
Eqs.~(\ref{Eq6}) and (\ref{Eq7}) are directly deduced from Eq.~(\ref{Eq3}).
\end{proof}
%
%
\begin{rem} 
If $\beta_0$ is a second-order random variable, then Eq.~(\ref{Eq6}) implies Eq.~(\ref{Eq7}). However, if Eq.~(\ref{Eq7}) holds for all $\bfx$ in $ D$, then this equation does not imply the existence of a positive random variable $\beta$ such that  Eq.~(\ref{Eq6}) is verified and \textit{a fortiori}, even if $\beta$ existed, this random variable would not be, in general,  a second-order random variable.
\end{rem}
\subsection{Representations of random field $[\bfK_0]$ as transformations of a non-Gaussian second-order symmetric matrix-valued random field $[\bfG]$}
\label{Section2.2}
Let $\curC^{2,S}_n$ be the set of all the second-order random fields $[\bfG] = \{[\bfG(\bfx)], \bfx\in  D\}$, defined on  probability space $(\Theta,\curT,\curP)$, with values in $\MM_n^{\rm S}(\RR)$. Consequently, for all $\bfx$ in $ D$, we have $E\{ \Vert\bfG(\bfx)\Vert^2_F \} < + \infty$.
%

In this section, we propose two representations of random field $[\bfK_0]$ belonging to $\curC^+_n$, yielding  the definition 
of two different subsets of $\curC^+_n$:
\begin{itemize}
\item \emph{Exponential type representation}.  The first type representation is written as $[\bfK_0(\bfx)] = \exp_\MM([\bfG(\bfx)])$ with $[\bfG]$ in $\curC^{2,S}_n$ and where $\exp_\MM$ denotes the exponential of symmetric square real matrices. It should be noted that random field $[\bfG]$ is not assumed to be Gaussian. If $[\bfG]$ were a Gaussian random field, then $[\bfK_0]$ would be a log-normal matrix-valued random field.

\item \emph{Square type representation}. The second type representation is written as  $[\bfK_0(\bfx)] = [\bfL(\bfx)]^T \, [\bfL(\bfx)]$ in which $[\bfL(\bfx)]$ is an upper triangular $(n\times n)$ real random matrix, for which the diagonal terms are positive-valued random variables, and  which is written as $[\bfL(\bfx)] = [\curL([\bfG(\bfx)])]$ with $[\bfG]$ in $\curC^{2,S}_n$ and where $[G] \mapsto [\curL([G])]$ is a well defined deterministic mapping from $\MM_n^{\rm S}(\RR)$ into the set of all the upper triangular $(n\times n)$ real deterministic matrices. Again, random field $[\bfG]$ is not assumed to be Gaussian. If for all $1\leq j\leq j'\leq n$, $[\bfG]_{jj'}$ are independent copies of a Gaussian random field and for a particular definition of mapping $[\curL([G])]$, then $[\bfK_0]$ would be the set $\SFGP$ of the Non-Gaussian matrix-valued random field previously introduced in \cite{Soize2006}.
\end{itemize}

\noindent These two representations are general enough, but the mathematical properties of each one will be slightly different and the computational aspects will be different.
\subsubsection{Exponential type representation of random field $[\bfK_0]$}
\label{Section2.2.1}

It should be noted that for all symmetric real matrix $[A]$ in $\MM_n^{\rm S}(\RR)$, $[B] = \exp_\MM([A])$ is a well defined matrix belonging to $\MM_n^+(\RR)$.
All matrix $[A]$ in $\MM_n^{\rm S}(\RR)$ can be written as $[A] = [\Phi]\, [\mu]\, [\Phi]^T$ in which $[\mu]$ is the diagonal matrix of the real eigenvalues $\mu_1,\ldots , \mu_n$ of $[A]$ and $[\Phi]$ is the orthogonal real matrix  made up of the associated eigenvectors. We then have $[B] = [\Phi]\,\exp_\MM([\mu])[\Phi]^T$ in which
$\exp_\MM([\mu])$ is the diagonal matrix in $\MM_n^+(\RR)$ such that $[\exp_\MM([\mu])]_{jk} = e^{\mu_j}\, \delta_{jk}$ .

\par 
For all random field $[\bfG]$ belonging to $\curC^{2,S}_n$, the random field $[\bfK_0] = \exp_\MM([\bfG])$
defined, for all $\bfx$ in $ D$, by
\begin{equation}                                                                                                                               \label{Eq8}
[\bfK_0(\bfx)] = \exp_\MM([\bfG(\bfx)])  \, ,
\end{equation}
belongs to $\curC^+_n$. If $[\bfK_0]$ is any random field given in $\curC^+_n$, then there exists a unique random field $[\bfG]$ with values in the
$\MM^{\rm S}_n(\RR)$ such that, for all $\bfx$ in $ D$,
\begin{equation}                                                                                                                               \label{Eq9}
[\bfG(\bfx)] = \log_\MM([\bfK_0(\bfx)])  \, .
\end{equation}
in which $\log_\MM$ is the reciprocity mapping of $\exp_\MM$ which is defined on $\MM^+_n(\RR)$ with values in $\MM^{\rm S}_n(\RR)$, but in general, this random field $[\bfG]$ is not a second-order random field and therefore, is not in $\curC^{2,S}_n$.
The following lemma shows that, if a random field $[\bfK_0]$ in $\curC^+_n$, satisfies additional properties, then 
there exists $[\bfG]$ in $\curC^{2,S}_n$ such that $[\bfK_0] = \exp_\MM([\bfG])$.
\begin{lemma}\label{lemma3} 
Let $[\bfK_0]$ be a random field belonging to $\curC^+_n$ such that for all $\bfx \in  D$,
\begin{equation}                                                                                                                               \label{Eq10}
E\{\Vert\bfK_0(\bfx)\Vert_F^2\} < + \infty \quad , \quad E\{\Vert[\bfK_0(\bfx)]^{-1}\Vert_F^2\} < + \infty  \, .
\end{equation}
Then there exists $[\bfG]$ belonging to $\curC^{2,S}_n$ such that $[\bfK_0] = \exp_\MM([\bfG])$.
\end{lemma}
%
%
\begin{proof}
For $\bfx$ fixed in $ D$, we use the spectral representation of  $[\bfK_0(\bfx)]$ introduced in the proof of Lemma \ref{lemma1}.
 The hypotheses introduced in Lemma \ref{lemma3} can be rewritten as
$E\{\Vert\bfK_0(\bfx)\Vert_F^2\}        = \sum_{j=1}^n E\{\Lambda_j(\bfx)^2 \} < + \infty$ and
$E\{\Vert[\bfK_0(\bfx)]^{-1}\Vert_F^2\}$ $ = $ $\sum_{j=1}^n E\{\Lambda_j(\bfx)^{-2} \} < + \infty$.
Similarly, it can easily be proven that
$E\{\Vert\bfG(\bfx)\Vert_F^2\}$ $= $ $\sum_{j=1}^n E\{ ( \log\Lambda_j(\bfx) )^2 \}$. Since $n$ is finite, the  proof of the lemma will be complete if we prove that, for all $j$, we have $E\{ ( \log\Lambda_j(\bfx) )^2 \} < +\infty$ knowing that
$E\{\Lambda_j(\bfx)^{-2} \} < +\infty$ and $E\{\Lambda_j(\bfx)^2 \}< +\infty$.
For $j$ and $\bfx$ fixed, let $P(d\lambda)$ be the probability distribution of the random variable $\Lambda_j(\bfx)$. We have
$E\{ ( \log\Lambda_j(\bfx) )^2 \} = \int_0^1 (\log\lambda)^2\, P(d\lambda) + \int_1^{+\infty} (\log\lambda)^2\, P(d\lambda)$. For $0 < \lambda < 1$, we have $(\log\lambda)^2 < \lambda^{-2}$, and for  $1 < \lambda < +\infty$, we have $(\log\lambda)^2 < \lambda^{2}$. It can then be deduced that
$E\{ ( \log\Lambda_j(\bfx) )^2 \} < \int_0^1 \lambda^{-2}\, P(d\lambda) + \int_1^{+\infty} \lambda^2\, P(d\lambda)$
$ < \int_0^{+\infty} \lambda^{-2}\, P(d\lambda) + \int_0^{+\infty} \lambda^2\, P(d\lambda)$
$ =  E\{\Lambda_j(\bfx)^{-2} \}  + E\{\Lambda_j(\bfx)^2 \} < +\infty$.
\end{proof}
\begin{rem} 
The converse of Lemma \ref{lemma3} does not hold. If $[\bfG]$ is any random field in $\curC^{2,S}_n$, in general, the random field
$[\bfK_0] = \exp_\MM([\bfG])$ is not a second-order random field.
\end{rem}
\begin{proposition}\label{proposition1} 
Let $[\bfG]$ be a random field belonging to $\curC^{2,S}_n$ and let $[\bfK_0]$ be the random field belonging to
$\curC^+_n$ such that, for all $\bfx$ in $ D$,  $[\bfK_0(\bfx)] = \exp_\MM([\bfG(\bfx)])$. Then the following results hold:
\begin{enumerate}[(ii)]\renewcommand{\theenumi}{\roman{enumi}}
\item  For all $\bfx$ in $ D$,
\begin{align}
 & \Vert\bfK_0(\bfx)\Vert_F \leq \sqrt{n}\, e^{\Vert\bfG(\bfx)\Vert_F} \quad a.s\, ,                                            \label{Eq11} \\
 & E\{\Vert\bfK_0(\bfx)\Vert_F^2\} \leq n\, E\{e^{2\Vert\bfG(\bfx)\Vert_F}\}\, ,                                               \label{Eq12} \\
 & E\{\Vert\log_\MM[\bfK_0(\bfx)]\Vert_F^2\}  < + \infty \, .                                                                   \label{Eq13}
\end{align}
\item If $\Vert\bfG(\bfx)\Vert_F \leq \beta_G < +\infty$ almost surely, in which $\beta_G$ is a positive random variable independent of $\bfx$, then 
$\Vert\bfK_0(\bfx)\Vert_F \leq \beta_0 < +\infty$ almost surely, in which $\beta_0$ is the positive random variable, independent of $\bfx$, such that
$\beta_0= \sqrt{n}\, e^{\beta_G}$.
\end{enumerate}
\end{proposition}

\begin{proof}
For (i), we have
$\Vert \exp_\MM(\bfG(\bfx))\Vert_{2} \leq  e^{\Vert\bfG(\bfx)\Vert_{2}}$. Since $[\bfK_0(\bfx)] = \exp_\MM(\bfG(\bfx))$ and since
$\Vert\bfK_0(\bfx)\Vert_F \leq \sqrt{n}\, \Vert\bfK_0(\bfx)\Vert_{2}$, it can then be deduced that
$\Vert\bfK_0(\bfx)\Vert_F \leq \sqrt{n}\, e^{\Vert\bfG(\bfx)\Vert_F}$ and then
$E\{\Vert\bfK_0(\bfx)\Vert_F^2\} \leq n\, E\{e^{2\Vert\bfG(\bfx)\Vert_F}\}$. Since $[\bfG]$ is a second-order random field, the inequality \eqref{Eq13}  is deduced from Eq.~(\ref{Eq9}).
For (ii), the proof is directly deduced from the inequality \eqref{Eq11}.
\end{proof}
\subsubsection{Square type representation of random field $[\bfK_0]$}
\label{Section2.2.2}

 Let $g\mapsto h(g;a)$ be a  given function from $\RR$ in $\RR^+$, depending on one positive real parameter $a$. For all fixed $a$,
it is assumed that:

\begin{enumerate}[(ii)]
\renewcommand{\theenumi}{\roman{enumi}}
\item  $h(.;a)$ is a strictly monotonically increasing function on $\RR$, which means that $h(g;a) < h(g';a)$ if $-\infty < g < g' < +\infty$;
\item there are  real numbers $0 < c_h < +\infty$ and   $0 < c_a< +\infty$, such that, for all $g$ in $\RR$, we have $h(g;a) \leq  c_a + c_h\, g^2$.
\end{enumerate}

\noindent It should be noted that $c_a \geq h(0,a)$. In addition, from (i)  it can be deduced that, for $g < 0$, we have $h(g;a) <  h(0,a) \leq  c_a  < c_a + c_h\, g^2$. Therefore, the inequality given in (ii), which is true for $g <0$, allows the behavior of the increasing function $g\mapsto h(g;a)$ to be controlled for $g > 0$. Finally, the introduced hypotheses imply that, for all $a >0$, $g\mapsto h(g;a)$ is a one-to-one mapping from $\RR$ onto $\RR^+$ and consequently, the reciprocity mapping, $v\mapsto h^{-1}(v;a)$, is a strictly monotonically increasing function from $\RR^+$ onto $\RR$.\\

\noindent The square type representation of random field $[\bfK_0]$ belonging to $\curC^{+}_n$ is then defined as follows.
For all $\bfx$ in $ D$,
\begin{equation}                                                                                                                               \label{Eq13bis}
[\bfK_0(\bfx)] = \LL([\bfG(\bfx)])  \, ,
\end{equation}
in which  $\{[\bfG(\bfx)],\bfx\in D\}$ is a random field belonging to $\curC^{2,S}_n$ and where $[G]\mapsto \LL([G])$ is a measurable mapping from
$\MM_n^{\rm S}(\RR)$  into $\MM_n^{+}(\RR)$ which is defined as follows. The matrix $[K_0]=\LL([G])\in \MM_n^{+}(\RR)$ is written as
\begin{equation}                                                                                                                               \label{Eq14}
[K_0] = [L]^T \, [L]  \, ,
\end{equation}
in which $[L]$ is an upper triangular $(n\times n)$ real matrix with positive diagonal, which is written as
\begin{equation}                                                                                                                               \label{Eq15}
[L] = \curL([G])  \, ,
\end{equation}
 where $[G] \mapsto \curL([G])$ is the measurable  mapping from $\MM_n^{\rm S}(\RR)$ into $\MM_n^{\rm U}(\RR)$ defined by
\begin{align}                                                                                                                               \label{Eq16}
&[\curL([G])]_{jj'}  = [G]_{jj'} \,\, , \,\, 1 \leq j < j' \leq n \, ,
%
%
\\
 \label{Eq17}
&[\curL([G])]_{jj}  = \sqrt{h([G]_{jj};a_j)} \,\, , \,\, 1 \leq j \leq n \, ,
\end{align}
in which $a_1,\ldots , a_n$ are positive real numbers.

 If $[\bfK_0]$ is any random field given in $\curC^+_n$, then there exists a unique random field $[\bfG]$ with values in the $\MM^{\rm S}_n(\RR)$ such that, for all $\bfx$ in $ D$,
\begin{equation}                                                                                                                               \label{Eq17bis}
[\bfG(\bfx)] = \LL^{-1}([\bfK_0(\bfx)])  \, ,
\end{equation}
in which $\LL^{-1}$ is the reciprocity function of $\LL$, from  $\MM_n^{+}(\RR)$ into $\MM_n^{\rm S}(\RR)$, which is explicitly defined as follows.
For all $1 \leq j \leq j' \leq n$,
\begin{equation}                                                                                                                               \label{Eq17ter}
[\bfG(\bfx)]_{jj'} = [\curL^{-1}([\bfL(\bfx)])]_{jj'}  \quad , \quad [\bfG(\bfx)]_{j'j} = [\bfG(\bfx)]_{jj'} \, .
\end{equation}
in which $[L] \mapsto \curL^{-1}([L])$ is the unique reciprocity mapping of $\curL$ (due to the existence of $v\mapsto h^{-1}(v;a)$) defined on $\MM_n^{\rm U}(\RR)$, and where $[\bfL(\bfx)]$ follows from the Cholesky factorization of random matrix $[\bfK_0(\bfx)] = [\bfL(\bfx)]^T\,[\bfL(\bfx)] $ (see Eqs.~\eqref{Eq13bis} and \eqref{Eq14}).

\paragraph*{Example of function $h$}
Let us give an example which shows that there exists at least one such a construction. For instance, we can choose  $h=h^\APM$, in which the function $h^\APM$ is defined in \cite{Soize2006} as follows. Let be $s=\delta/\sqrt{n+1}$ in which  $\delta$ is a parameter  such that $0 < \delta < \sqrt{(n+1)/(n-1)}$ and which allows the statistical fluctuations level to be controlled. Let be  $a_j = 1/(2\,s^2) + (1-j)/2 > 0$ and  $h^\APM(g;a)  = 2\, s^2\, F^{-1}_{\Gamma_a}(F_W(g/s))$ with $F_W(\widetilde w) =\int_{-\infty}^{\widetilde w} \frac{1}{\sqrt{2\pi}} \, \exp(- \frac{1}{2}w^2)\, dw$ and  $F^{-1}_{\Gamma_a}(u) = \gamma$ the reciprocal function such that $F_{\Gamma_a}(\gamma) = u$ with $F_{\Gamma_a}(\gamma) = \int_0^\gamma \frac{1}{\Gamma(a)}\, t^{a-1}\, e^{-t}\, dt$ and
$\Gamma(a) =\int_0^{+\infty}  t^{a-1}\, e^{-t}\, dt$. Then, for all $j=1,\ldots , n$, it can be proven \cite{Soize2006} that $g\mapsto h^\APM(g;a_j)$ is a strictly monotonically increasing function from $\RR$ into $\RR^+$ and there are positive real numbers $c_h$ and $c_{a_j}$ such that, for all $g$ in $\RR$, we have $h^\APM(g;a_j) \leq c_{a_j} + c_h\, g^2$. In addition, it can easily be seen that the reciprocity function is written as ${h^\APM}^{-1}(v;a)  = s \, F_W^{-1}( F_{\Gamma_a}(v/(2s^2))$.
\begin{proposition}\label{proposition2} 
Let $[\bfG]$ be a random field belonging to $\curC^{2,S}_n$ and let $[\bfK_0]$ be the random field belonging to
$\curC^+_n$ such that, for all $\bfx$ in $ D$,  $[\bfK_0(\bfx)] = \LL([\bfG(\bfx)])$. We then have the following results:
\begin{enumerate}[(ii)]\renewcommand{\theenumi}{\roman{enumi}}
\item There exist two real numbers $0 < \gamma_0 < +\infty$ and $0 < \gamma_1 < +\infty$ such that, for all $\bfx$ in $ D$, we have
\begin{align}
 & \Vert\bfK_0(\bfx)\Vert_F \leq \gamma_0 + \gamma_1\Vert \bfG(\bfx)\Vert_F^2 \quad a.s\, ,                                         \label{Eq18} \\
 & E\{\Vert\bfK_0(\bfx)\Vert_F\} < +\infty\, ,                                                                                      \label{Eq19} \\
 & \hbox{If} \,\, E\{\Vert\bfG(\bfx)\Vert_F^4\} < +\infty\, , \,\, \hbox{then} \,\, E\{\Vert\bfK_0(\bfx)\Vert_F^2\} < +\infty \, .     \label{Eq20}
\end{align}
\item If $\Vert\bfG(\bfx)\Vert_F \leq \beta_G < +\infty$ a.s, in which $\beta_G$ is a positive random variable independent of $\bfx$, we then have
$\Vert\bfK_0(\bfx)\Vert_F \leq \beta_0 < +\infty$ almost surely, in which $\beta_0$ is the positive random variable, independent of $\bfx$, such that
$\beta_0= \gamma_0 +\gamma_1\,\beta_G^2$.
\end{enumerate}
\end{proposition}
\begin{proof} 
(i) Since $[\bfK_0(\bfx)] = [\bfL(\bfx)]^T\, [\bfL(\bfx)]$, we have $\Vert\bfK_0(\bfx)\Vert_F\leq \Vert\bfL(\bfx)\Vert_F^2$ with
$\Vert\bfL(\bfx)\Vert_F^2 = \sum_j h([\bfG(\bfx)]_{jj}; a_j) +\sum_{j<j'}[\bfG(\bfx)]_{jj'}^2$.
Therefore, $\Vert\bfK_0(\bfx)\Vert_F\leq \sum_j c_{a_j} +\max\{c_h, 1/2\}$ $\, \sum_{j,j'}[\bfG(\bfx)]_{jj'}^2$ which yields inequality \eqref{Eq18} with
$\gamma_0 = \sum_j c_{a_j}$ and $\gamma_1 =\max\{c_h, 1/2\}$. Taking the mathematical expectation yields $E\{\Vert\bfK_0(\bfx)\Vert_F\} \leq
\gamma_0 + \gamma_1 E\{\Vert \bfG(\bfx)\Vert_F^2\} < +\infty$ because $[\bfG]$ is a second-order random field. Taking the square and then the mathematical expectation of Eq.~(\ref{Eq18}) yields Eq.~(\ref{Eq20}).
\\(ii) The proof is directly deduced from Eq.~(\ref{Eq18}).
\end{proof}
\subsection{Representation of any random field $[\bfG]$ in $\curC^{2,S}_n$}
\label{Section2.3}

As previously explained, we are interested in the identification of the random field $[\bfK]$ which belongs to  $\curC^+_n$, by solving a statistical inverse problem related to a stochastic boundary value problem. Such an identification is carried out using the proposed representation of $[\bfK]$ (defined by Eq.~(\ref{Eq2})) as a function of the random field $[\bfK_0]$ which is written either as $[\bfK_0(\bfx)] = \exp_\MM([\bfG(\bfx)])$ (see Eq.~(\ref{Eq8})) or as $[\bfK_0(\bfx)] = [\curL([\bfG(\bfx)])]^T \, [\curL([\bfG(\bfx)])]$ (see Eqs.~(\ref{Eq14}) and (\ref{Eq15})). In these two representations, $[\bfG]$ is any random field in $\curC^{2,S}_n$, which has to be identified (instead of $[\bfK]$ in $\curC^+_n$). Consequently, we have to construct a representation of any random field $[\bfG]$ in $\curC^{2,S}_n$. Since any $[\bfG]$ in $\curC^{2,S}_n$ is a second-order random field, a general representation of $[\bfG]$, adapted to its identification, is the polynomial chaos expansion for which the coefficients of the expansion constitutes a family of functions from $ D$ into $\MM_n^{\rm S}(\RR)$.\par

 It is well known that a direct identification of such a family of matrix-valued functions cannot easily be done. An adapted representation must be introduced consisting in choosing a deterministic vector basis, then representing $[\bfG]$ on this deterministic vector basis and finally, performing a polynomial chaos expansion of the random coefficients on this deterministic vector basis. The representation is generally in high stochastic dimension. Consequently, the identification of a very large number of coefficients must be done. It is then interesting to use a statistical reduction and thus to choose, for the deterministic vector basis, the  Karhunen-Lo\`eve vector basis. Such a vector basis is constituted of the family of the eigenfunctions of the compact covariance operator of random field $[\bfG]$.
\subsubsection{Covariance operator of random field $[\bfG]$ and eigenvalue problem}
\label{Section2.3.1}

Let $[G_0(\bfx)] = E\{ [\bfG(\bfx)]\}$ be the mean function of $[\bfG]$, defined on $ D$ with values in $\MM_n^{\rm S}(\RR)$.
The covariance function of random field $[\bfG]$ is the function $(\bfx,\bfx')\mapsto C_\bfG(\bfx,\bfx')$, defined on $ D\times D$, with values in the space $\MM_n^{\rm S}(\RR)\otimes\MM_n^{\rm S}(\RR)$ of fourth-order tensors, such that
\begin{equation}                                                                                                                               \label{Eq21}
C_\bfG(\bfx,\bfx') = E\{ ([\bfG(\bfx)] - [G_0(\bfx)] ) \otimes ([\bfG(\bfx')] - [G_0(\bfx')] ) \} \, .
\end{equation}
It is assumed that $[G_{0}]$ is a uniformly bounded function on  $ D$, i.e.
\begin{equation}
\Vert G_{0}\Vert_{\infty} := \textrm{ess sup}_{\bfx\in D} \Vert G_0(\bfx) \Vert_F  < +\infty
\end{equation}
 and that function $C_\bfG$ is square integrable on $D\times D$. Consequently, the covariance operator $\hbox{Cov}_\bfG$, defined by the kernel $C_\bfG$, is a Hilbert-Schmidt, symmetric, positive operator in the Hilbert space $L^2(D,\MM_n^{\rm S}(\RR))$ equipped with the inner product,
\begin{equation}                                                                                                                               \label{Eq22}
\ll [G_i] \, , [G_j]\gg = \int_D  \trmat\{[G_i(\bfx)]^T \, [G_j(\bfx)]\}\,d\bfx \, ,
\end{equation}
and the associated norm $ |||[G_i] ||| = \ll [G_i] \, , [G_i]\gg^{1/2}$.
The eigenvalue problem related to the covariance operator $\hbox{Cov}_\bfG$, consists in finding the family $\{[G_i(\bfx)] ,\bfx\in D\}_{i\geq 1}$ of the normalized eigenfunctions with values in $\MM_n^{\rm S}(\RR)$ and the associated positive eigenvalues $\{\sigma_i\}_{i\geq 1}$ with $\sigma_1\geq \sigma_2\geq \ldots \rightarrow 0$ and
$\sum_{i=1}^{+\infty} \sigma_i^2 < +\infty$, such that
\begin{equation}                                                                                                                               \label{Eq23}
\int_D C_\bfG(\bfx,\bfx') :  [G_i(\bfx')] \, d\bfx' =\sigma_i \, [G_i(\bfx)]  \ , \quad i\geq 1 \, ,
\end{equation}
in which $\{C_\bfG(\bfx,\bfx') :  [G_i(\bfx')]\}_{k\ell} = \sum_{k'\ell'} \{C_\bfG(\bfx,\bfx')\}_{k\ell k'\ell'}\, \{[G_i(\bfx')]\}_{k'\ell'}$.
The normalized family $\{[G_i(\bfx)] ,\bfx\in D\}_{i\geq 1}$ is a Hilbertian basis of $L^2(D,\MM_n^{\rm S}(\RR))$ and consequently,
$\ll [G_i] \, , [G_j]\gg = \delta_{ij}$. It can easily be verified that
\begin{equation}                                                                                                                               \label{Eq24}
E\{\Vert \bfG(\bfx)\Vert^2_F\} = \Vert G_0(\bfx) \Vert^2_F + \sum_{i=1}^{+\infty} \sigma_i\, \Vert G_i(\bfx) \Vert^2_F < +\infty \ , \quad  \forall\bfx \in D \, ,
\end{equation}
\begin{equation}                                                                                                                               \label{Eq25}
\int_D E\{\Vert \bfG(\bfx)\Vert^2_F\}\, d\bfx = |||G_0|||^2 + \sum_{i=1}^{+\infty} \sigma_i < +\infty  \, .
\end{equation}
We now give some properties which will be useful later.
Since $C_\bfG$ is a covariance function, we have $||C_\bfG(\bfx,\bfx')||^2 \leq \trmat\{C_\bfG(\bfx,\bfx) \} \times$ $\trmat\{C_\bfG(\bfx',\bfx') \}$ in which $||C_\bfG(\bfx,\bfx')||^2 = \sum_{k \ell k'\ell'} \{C_\bfG(\bfx,\bfx')\}_{k\ell k'\ell'}^2$ and
$\trmat\{C_\bfG(\bfx,\bfx) \}$ $ = \sum_{k \ell} \{C_\bfG(\bfx,\bfx)\}_{k\ell k\ell}$ $\geq 0$. We can then deduce the following Lemma.
\begin{lemma}\label{lemma4} 
If $\bfx\mapsto\trmat\{C_\bfG(\bfx,\bfx) \}$ is integrable on $D$, then $C_\bfG$ is square integrable on $D\times D$.
If $\bfx\mapsto\trmat\{C_\bfG(\bfx,\bfx) \}$ is bounded on $ D$, then $C_\bfG$ is bounded on $ D\times D$ (and thus square integrable on $D\times D$) and the eigenfunctions $\bfx \mapsto [G_i(\bfx)]$ are bounded functions from $ D$ into $\MM_n^{\rm S}(\RR)$. For $i\geq 1$, we then have $\Vert G_i \Vert_\infty = \textrm{ess sup}_{\bfx\in D} \Vert G_i(\bfx) \Vert_F  < +\infty$.
\end{lemma}
Using additional assumptions about the kernel $C_\bfG$, more results concerning the decrease rate of eigenvalues $\{\sigma_i\}_{i\geq 1}$ can be obtained.
\begin{lemma} \label{lemma5} 
We have the following results for $d=1$ and for $d\geq2$.
\begin{enumerate}
\item[(a)]
For $d=1$, it is proven (see Theorem 9.1 of Chapter II of \cite{Zabreyko1975}) that, if for a given integer $\mu \geq 1$ and for all integer $\alpha$ such that $1\leq \alpha\leq \mu$, the functions $(x,x')\mapsto \partial^\alpha C_\bfG(x,x')/ \partial {x'}^\alpha$ are bounded functions on $ D\times D$, then $\sum_{i=1}^{+\infty} \sigma_i^{\zeta + 2/(2\mu +1)} < +\infty$ for any $\zeta >0$.

\item[(b)]
For $d\geq2$ (finite integer), a useful result is given by Theorem 4 of \cite{Kuhn1987}. Let us assume that $ D$ has a sufficiently smooth boundary $\partial D$.  For $\bfx$ in $D$, let $\bfh$ be such that $\bfx+\bfh$ and $\bfx+2 \bfh$ belong to $D$.  We then define the difference operator $\Delta_\bfh$ such that, for all $\bfx'$ fixed in $ D$, $\Delta_\bfh C_\bfG(\bfx,\bfx')= C_\bfG(\bfx +\bfh,\bfx')-C_\bfG(\bfx,\bfx')$ and $\Delta_\bfh^2$ is defined as the second iterate of $\Delta_\bfh$. For all $\bfx'$ fixed in $ D$, for a given positive integer $\mu$ such that $\mu \geq 1$ and for a given real $\zeta$ such that $0 < \zeta \leq 1$, let be
\begin{equation}
\nonumber
\Vert C_\bfG(\cdot,\bfx')\Vert_{\mu} = \!\!  \sum_{\vert\bfalpha\vert \leq \mu} \, \sup_{\bfx\in D}\Vert\curD^\bfalpha C_\bfG(\bfx,\bfx')\Vert
+ \sum_{\vert\bfalpha\vert =  \mu} \, \sup_{\bfx,\bfh} \frac {\Vert \Delta_\bfh^2 \curD^\bfalpha C_\bfG(\bfx,\bfx')\Vert}{\Vert \bfh\Vert^{\mu+\zeta}} \, ,
\end{equation}
in which  $\curD^\bfalpha C_\bfG(\bfx,\bfx')= \partial^{\vert\bfalpha\vert} C_\bfG(\bfx,\bfx') / \partial x_1^{\alpha_1}\ldots \partial x_d^{\alpha_d}$ with $\vert\bfalpha\vert = \alpha_1 + \ldots + \alpha_d$. Therefore, if $C_\bfG$ is continuous on $ D\times D$ such that
%
$\sup_{\bfx'\in D} \Vert C_\bfG(\cdot,\bfx')\Vert_{\mu} < +\infty,$
%
we then have the following decrease of eigenvalues:
%
$\sigma_i = O( i^{-1 -(\mu+\zeta)/d}).$
%
\item[(c)]
For $d\geq2$ (finite integer), a similar result to (b) in Lemma \ref{lemma5} can be found in \cite{Menegatto2012} under hypotheses weaker than those introduced above, but in practice, it seems much more difficult to verify these hypotheses for a given kernel $C_\bfG$.
\end{enumerate}
\end{lemma}

%
\paragraph*{Comment about the eigenvalue problem}
Any symmetric $(n\times n)$ real matrix $[G]$ can be represented by a real vector $\bfw$ of dimension $n_\bfw=n(n+1)/2$ such that $[G] = \curG(\bfw)$ in which $\curG$ is the one-to-one linear mapping from $\RR^{n_\bfw}$ in $\MM_n^{\rm S}(\RR)$, such that, for $1\leq i\leq j \leq n$, one has $[G]_{ij} = [G]_{ji}= w_k$ with $k=i +j(j-1)/2$. Let $\curL^{2}_{n_\bfw}$ be  the set of all the $\RR^{n_\bfw}$-valued second-order random fields $\{\bfW(\bfx), \bfx\in D\}$ defined on probability space $(\Theta,\curT,\curP)$. Consequently, any random field $[\bfG]$ in $\curC^{2,S}_n$ can be written as
\begin{equation}                                                                                                                               \label{Eq26}
[\bfG] = \curG(\bfW), \quad \bfW = \curG^{-1}([\bfG])  ,
\end{equation}
in which $\bfW$ is a random field in $\curL^{2}_{n_\bfw}$. The eigenvalue problem defined by Eq.~(\ref{Eq23}) can be rewritten as
\begin{equation}                                                                                                                               \label{Eq27}
\int_D [C_\bfW(\bfx,\bfx')] \, \bfw^i(\bfx') \, d\bfx' =\sigma_i \, \bfw^i(\bfx)  \  , \quad i\geq 1 ,
\end{equation}
in which $[C_\bfW(\bfx,\bfx')] = E\{(\bfW(\bfx) - \bfw^0(\bfx)) (\bfW(\bfx') - \bfw^0(\bfx'))^T\}$ and where
$\bfw^0 = \curG^{-1}([G_0])$. For $i\geq 1$, the eigenfunctions $[G_i]$ are then given by $[G_i] = \curG(\bfw^i)$.
\subsubsection{Chaos representation of random field $[\bfG]$}
\label{Section2.3.2}

Under the hypotheses introduced in Section~\ref{Section2.3.1},  random field $[\bfG]$ admits the following Karhunen-Lo\`eve decomposition,
\begin{equation}                                                                                                                               \label{Eq28}
[\bfG(\bfx)] = [G_0(\bfx)] + \sum_{i=1}^{+\infty} \sqrt{\sigma_i}\, [G_i(\bfx)] \, \eta_i\, ,
\end{equation}
in which $\{\eta_i\}_{i\geq 1}$ are uncorrelated random variables with zero mean and unit variance. The second-order random variables  $\{\eta_i\}_{i\geq 1}$ are represented using the following polynomial chaos expansion
\begin{equation}                                                                                                                               \label{Eq29}
\eta_i =  \sum_{j=1}^{+\infty} y_i^j \, \Psi_j(\{\Xi_k\}_{k\in\NN}) \, ,
\end{equation}
in which $\{\Xi_k\}_{k\in\NN}$ is a countable set of independent normalized Gaussian random variables
and where $\{\Psi_j\}_{j\geq 1}$ is the polynomial chaos basis composed of normalized multivariate Hermite
polynomials such that $E\{\Psi_j(\{\Xi_k\}_{k\in\NN})\, \Psi_{j'}(\{\Xi_k\}_{k\in\NN})\}$ $ = \delta_{jj'}$.  Since $E\{\eta_i\eta_{i'}\} = \delta_{ii'}$, it can be deduced that
\begin{equation}                                                                                                                               \label{Eq30}
 \sum_{j=1}^{+\infty} y_i^j \, y_{i'}^j  = \delta_{ii'}\, .
\end{equation}

\subsection{Random upper bound for random field  $[\bfK]$}
\label{Section2.4}
\ 
\begin{lemma}\label{lemma6} 
If $\sum_{i=1}^{+\infty} \sqrt{\sigma_i} \, \Vert G_i \Vert_\infty < +\infty$, then for all $\bfx$ in $ D$, 
\begin{equation}
\Vert \bfG(\bfx) \Vert_F \leq \beta_G < +\infty \quad \text{a.s.},
\end{equation} 
in which  $\beta_G$ is the second-order positive-valued random variable,
\begin{equation}                                                                                                                               \label{Eq31}
\beta_G = \Vert G_0\Vert_\infty + \sum_{i=1}^{+\infty} \sqrt{\sigma_i} \, \Vert G_i \Vert_\infty \, \vert \eta_i \vert  ,
\quad E\{ \beta_G^2\}  < +\infty \, .
\end{equation}
\end{lemma}
%
%
\begin{proof}The expression of $\beta_G$ defined in Eq.~(\ref{Eq31}) is directly deduced from Eq.~(\ref{Eq28}). The random variable $\beta_G$ can be written as
$\beta_G = \Vert G_0\Vert_\infty + \widehat \beta$. Clearly, if $E\{\widehat \beta^2\} = \sum_{i,i'} \sqrt{\sigma_i}\sqrt{\sigma_{i'}} \Vert G_i \Vert_\infty \Vert G_{i'} \Vert_\infty E\{\vert\eta_i\vert \, \vert\eta_{i'}\vert\} < +\infty$, then $E\{ \beta_G^2\}  < +\infty$.
We have $E\{\vert\eta_i\vert \, \vert\eta_{i'}\vert\} \leq \sqrt{E\{\eta_i^2\}}\sqrt{E\{\eta_{i'}^2\}} =1$ and thus
$E\{\widehat \beta^2\} \leq (\sum_{i=1}^{+\infty} \sqrt{\sigma_i} \, \Vert G_i \Vert_\infty  )^2$.
\end{proof}

\begin{rem} 
It should be noted that, if $\Vert G_i \Vert_\infty < c_\infty < +\infty$ for all $i \geq 1$ with $c_\infty$ independent of $i$,
we then have $\sum_{i=1}^{+\infty} \sqrt{\sigma_i} \, \Vert G_i \Vert_\infty < +\infty$ if:
\begin{enumerate}[(ii)]\renewcommand{\theenumi}{\roman{enumi}}

\item for $d=1$, the hypothesis of  (a) in Lemma \ref{lemma5} holds for $\mu=2$. The proof is the following. We have $\sum_{i=1}^{+\infty} \sqrt{\sigma_i} \, \Vert G_i \Vert_\infty < c_\infty \sum_{i=1}^{+\infty} \sqrt{\sigma_i}$. Since $\sigma_i \rightarrow 0$ for $i \rightarrow +\infty$, there exists an integer $i_0 \geq 1$ such that, for all $i \geq i_0$, we have $\sigma_i < 1$. For $\mu=2$, then there exits  $0 < \zeta < 1/10$ such that $\sum_{i=i_0}^{+\infty} \sqrt{\sigma_i} \leq \sum_{i=i_0}^{+\infty} \sigma_i^{\zeta+2/5} < +\infty$, which yields
 $\sum_{i=1}^{+\infty} \sqrt{\sigma_i}  < +\infty$, and consequently, $\sum_{i=1}^{+\infty} c_\infty\, \sqrt{\sigma_i}  < +\infty$.

\item for $d\geq 2$ (finite integer), the hypothesis of  (b) in Lemma \ref{lemma5} holds for $\mu=d$. The proof is then the following. We have $\sum_{i=1}^{+\infty} \sqrt{\sigma_i} \, \Vert G_i \Vert_\infty < c_\infty \sum_{i=1}^{+\infty} \sqrt{\sigma_i}$ and for $\mu=d$, $\sqrt{\sigma_i} = O( i^{-1 -\zeta/(2d)})$ and consequently, for
    $0 < \zeta \leq 1$, $\sum_{i=1}^{+\infty} \sqrt{\sigma_i}  < +\infty$.
\end{enumerate}
\end{rem}
The previous results allow the following proposition to be proven.
\begin{proposition}\label{proposition3} 
 Let $\sigma_i$ and $[G_i]$ be defined in Section~\ref{Section2.3.1}. If $\sum_{i=1}^{+\infty} \sqrt{\sigma_i} \, \Vert G_i \Vert_\infty < +\infty$, then for all $\bfx$ in $ D$,
\begin{equation}                                                                                                                               \label{Eq32}
\Vert\bfK(\bfx)\Vert_F \,\,\leq \,\, \beta < +\infty \quad a.s\, ,
\end{equation}
in which  $\beta$ is a positive-valued random variable independent of $\bfx$.
Let $\beta_G$ be the second-order positive-valued random variable  defined by Eq.~(\ref{Eq31}).

\begin{enumerate}[(ii)]\renewcommand{\theenumi}{\roman{enumi}}

\item \emph{(Exponential type representation)}  If $[\bfK(\bfx)]$ is represented by Eq.~(\ref{Eq2}) with Eq.~(\ref{Eq8}),
\begin{equation}                                                                                                                               \label{Eq33}
[\bfK(\bfx)] =\frac{1}{1+\varepsilon} [\underline L(\bfx)]^T\,\{ \varepsilon [\, I_n] + \exp_\MM([\bfG(\bfx)])\}  \, [\underline L(\bfx)] \, ,
\end{equation}
then, $\beta = \underline k_{\, 1} \sqrt{n}(\varepsilon + e^{\beta_G})/(1+\varepsilon)$.

\item \emph{(Square type representation)}  If $[\bfK(\bfx)]$ is represented by Eq.~(\ref{Eq2}) with Eqs.~(\ref{Eq14}) and (\ref{Eq15}),
\begin{equation}                                                                                                                               \label{Eq34}
[\bfK(\bfx)] =\frac{1}{1+\varepsilon} [\underline L(\bfx)]^T\,\{ \varepsilon [\, I_n]
+ [\curL([\bfG(\bfx)])]^T \, [\curL([\bfG(\bfx)])]\}  \, [\underline L(\bfx)] \, ,
\end{equation}
then, $\beta = \underline k_{\, 1} (\sqrt{n}\, \varepsilon + \gamma_0 +\gamma_1\,\beta_G^2)/(1+\varepsilon)$ in which
$\gamma_0$ and $\gamma_1$ are two positive and finite real numbers. In addition the random variable $\beta$ is such that
\begin{equation}                                                                                                                               \label{Eq35}
E\{\beta\} < +\infty \, .
\end{equation}
\end{enumerate}
\end{proposition}
%
%
\begin{proof}
Proposition~\ref{proposition3}  results from Lemma \ref{lemma2}, Propositions \ref{proposition1} and \ref{proposition2}, and Lemma \ref{lemma6}.
\end{proof}
\subsection{Approximation of random field $[\bfG]$}
\label{Section2.5}

Taking into account Eqs.~(\ref{Eq28}) to (\ref{Eq30}), we introduce the approximation $[\bfG^{(m,N)}]$ of the random field $[\bfG]$ such that,
\begin{align}                                                                                                                               \label{Eq36}
&[\bfG^{(m,N)}(\bfx)] = [G_0(\bfx)] + \sum_{i=1}^{m} \sqrt{\sigma_i}\, [G_i(\bfx)] \, \eta_i\, ,
\\                                                                                                                      
    \label{Eq37}
& \eta_i =  \sum_{j=1}^{N} y_i^j \, \Psi_j(\bfXi) \, ,
\end{align}
in which the $\{\Psi_{j}\}_{j=1}^N$ only depends on a random vector $\bfXi =(\Xi_1,\ldots ,\Xi_{N_g})$ of $N_g$ independent normalized Gaussian
random variables $\Xi_1,\ldots ,\Xi_{N_g}$ defined on  probability space $(\Theta,\curT,\curP)$. The
coefficients $y_{i}^j$ are supposed to verify $\sum_{j=1}^{N} y_i^j \, y_{i'}^j  = \delta_{ii'}$
which ensures that the random variables, $\{\eta_i\}_{i=1}^m$, are uncorrelated centered random variables with unit variance, which means that
$E\{\eta_i\eta_{i'}\} = \delta_{ii'}$. The relation between the coefficients can be rewritten as
\begin{equation}                                                                                                                               \label{Eq38}
[y]^T\, [ y] = [I_m] \, ,
\end{equation}
in which $[ y] \in \MM_{N,m}(\RR)$ is such that $[y]_{ji} =  y^j_{i}$ for $1\leq i \leq m$ and $1 \leq j \leq N$.
Introducing the random vectors $\bfeta = (\eta_1,\ldots , \eta_m)$ and $\bfPsi(\bfXi) = (\Psi_1(\bfXi),\ldots , \Psi_N(\bfXi))$, Eq.~\eqref{Eq37} can be rewritten as
\begin{equation}                                                                                                                               \label{Eq38bis}
\bfeta =  [y]^T\,  \bfPsi(\bfXi) \, .
\end{equation}
Equation~(\ref{Eq38}) means
that $[ y ]$ belongs to the compact Stiefel manifold
\begin{equation}                                                                                                                               \label{Eq39}
\VV_m(\RR^N) = \left\{ [y]\in \MM_{N,m}(\RR)\,  ; \; [y]^T \, [y] = [I_m] \right \} \, .
\end{equation}
With the above hypotheses, the covariance operators of random fields $[\bfG]$ and $[\bfG^{(m,N)}]$ coincide on the subspace spanned by the finite family $\{[G_i]\}_{i=1}^m$.
\section{Parametrization of discretized random fields in the general class and identification strategy}
\label{Section3}
Let us consider the approximation $\{[\bfG^{(m,N)}(\bfx)] ,\bfx\in D\}$ of $\{[\bfG(\bfx)] ,\bfx\in D\}$ defined by Eqs.~\eqref{Eq36} to \eqref{Eq38}. The corresponding approximation $\{[\bfK^{(m,N)}(\bfx)] ,\bfx\in D\}$ of random field $\{[\bfK(\bfx)] ,\bfx\in D\}$ defined by Eq.~\eqref{Eq33} or by Eq.~\eqref{Eq34}, is rewritten, for all $\bfx$ in $D$,  as
\begin{equation}                                                                                                                                \label{Eq40}
[\bfK^{(m,N)}(\bfx)] = \curK^{(m,N)}(\bfx,\bfXi,[y]) \, ,
\end{equation}
in which
$(\bfx,\bfy,[y]) \mapsto \curK^{(m,N)}(\bfx,\bfy,[y])$ is a mapping defined on $D\times\RR^{N_g}\times \VV_m(\RR^N)$ with values in $\MM_n^+(\RR)$.
\begin{itemize}
\item The first objective of this section is to constructed a parameterized general class of random fields, $\{[\bfK^{(m,N)}(\bfx)] ,\bfx\in D\}$, in introducing a minimal parametrization of the compact Stiefel manifold $\VV_m(\RR^N)$ with an algorithm of complexity $O(N m^2)$.
    \item The second objective will be the presentation of an identification strategy of an optimal random field $\{[\bfK^{(m,N)}(\bfx)] ,\bfx\in D\}$ in the parameterized general class, using partial and limited experimental data.
\end{itemize}
\subsection{Minimal parametrization of the compact Stiefel manifold $\VV_m(\RR^N)$}
\label{Section3.1}
Here, we introduce a particular minimal parametrization of the compact Stiefel manifold $\VV_m(\RR^N)$ using matrix exponentials (see e.g. \cite{Absil2004}). The dimension of $\VV_m(\RR^N)$  being  $ \nu = mN - m(m+1)/2$, the parametrization consists in introducing  a surjective mapping  from $\RR^\nu$ onto $\VV_m(\RR^N)$. The construction is as follows.

\begin{enumerate}[(iii)]\renewcommand{\theenumi}{\roman{enumi}}

\item 
Let $[a]$ be given in $\VV_m(\RR^N)$ and
let $[a_\perp]\in\MM_{N,N-m}(\RR)$ be the orthogonal complement of $[a]$, which is such that
$[a \,\,\, a_\perp]$ is in $\OO(N)$. Consequently, we have,
\begin{equation}                                                                                                                                \label{Eq41}
[a]^T\, [a] = [I_m], \qquad [a_\perp]^T\, [a_\perp] = [I_{N-m}] , \qquad [a_\perp]^T\, [a] = [0_{N-m,m}] \, .
\end{equation}
The columns of matrix $[a_\perp]$ can be chosen as the vectors of the orthonormal basis of the null space of $[a]^T$. In practice \cite{Golub1996}, $[a_\perp]$ can be constructed using the QR factorization of matrix $[a]= [Q_N]\, [R_m]$ and  $[a_\perp]$ is then made up of the columns $j=m+1,\ldots , N$ of $[Q_N]\in \OO(N)$.
%

\item   Let $[A]$ be a skew-symmetric $(m\times m)$ real matrix which then depends on $m(m-1)/2$ parameters denoted by $z_1,\ldots ,z_{m(m-1)/2}$ and such that, for $1\leq i < j \leq m$, one has $[A]_{ij} = - [A]_{ji} = z_k$ with $k=i+(j-1)(j-2)/2$ and for $1 \leq i\leq m$, $[A]_{ii} =0$.
\item  Let $[B]$ be a $((N-m)\times m)$ real matrix which then depends on $(N-m)m$ parameters denoted by
$z_{m(m-1)/2 + 1},\ldots , z_{m(m-1)/2 + (N-m)m }$ and such that, for $1\leq i\leq N-m$ and $1\leq j\leq m$, one has $[B]_{ij}= z_{m(m-1)/2 + k}$ with
$k=i+(j-1)(N-m)$.
Introducing $\bfz = (z_1,\ldots , z_\nu)$ in $\RR^\nu$, there is a one-to-one linear mapping $\curS$ from $\RR^\nu$ into
$\MM^{\rm SS}_m(\RR)\times \MM_{N-m,m}(\RR)$ such that
\begin{equation}                                                                                                                                \label{Eq42}
\{[A] , [B]\} = \curS(\bfz) \, ,
\end{equation}
in which matrices $[A]$ and $[B]$ are defined in (i) and (ii) above as a function of $\bfz$.
\end{enumerate}

Let $t > 0$ be a parameter which is assumed to be fixed. Then, for an arbitrary $[a]$ in $\VV_m(\RR^N)$, a first minimal parametrization of the compact Stiefel manifold  $\VV_m(\RR^N)$ can be defined by the mapping $\curM_{[a]}$ from $\RR^\nu$ onto $\VV_m(\RR^N)$ such that
\begin{equation}                                                                                                                                \label{Eq43}
[y] = \curM_{[a]}(\bfz) := [a \,\,\, a_\perp] \, \{\exp_\MM(t\, \left [
         \begin{array}{cc}
                     A     &    -B^T \\
                     B     &     0  \\
         \end{array}
                  \right ]) \} \, [I_{N,m}]\, .
\end{equation}
In the context of the present development, we are interested in the case for which $N\geq m$ and possibly, in the case for which $N\gg m$ with $N$ very large. The evaluation of the mapping $\curM_{[a]}$ defined  by Eq.~\eqref{Eq43} has a complexity $O(N^3)$.

We then propose to use a second form of minimal parametrization of the compact Stiefel with a reduced computational complexity.  This parametrization, which  is derived from the results presented in \cite{Edelman1998}, is defined by the mapping  $\curM_{[a]}$ from $\RR^\nu$ onto $\VV_m(\RR^N)$ such that
%
\begin{equation}                                                                                                                                \label{Eq44}
[y] = \curM_{[a]}(\bfz) := [a \,\,\, Q] \, \{\exp_\MM(t\, \left [
         \begin{array}{cc}
                     A     &    -R^T \\
                     R     &     0  \\
         \end{array}
                  \right ]) \} \, [I_{2m,m}]\, ,
\end{equation}
in which $[a\,\,\, Q]\in\MM_{N,2m}(\RR)$. The matrix $[Q]$ is in $\VV_m(\RR^N)$ and $[R]$ is an upper triangular $(m\times m)$ real matrix. These two matrices   are constructed using the QR factorization of the matrix $[a_\perp]\, [B]\in \MM_{N,m}(\RR)$,
\begin{equation}                                                                                                                                \label{Eq45}
[a_\perp]\, [B] = [Q] \, [R] \, .
\end{equation}
The evaluation of the mapping $\curM_{[a]}$ defined  by Eq.~\eqref{Eq44} has a complexity $O(Nm^2)$.

\subsection{Parameterized general class of random fields and parameterized random upper bound}
\label{Section3.2}

Using Eq.~\eqref{Eq44}, the parameterized general class  of random field $\{[\bfK^{(m,N)}(\bfx)] ,$ $\bfx\in D\}$ is then defined as
\begin{equation}                                                                                                                                \label{Eq46}
\curK := \{ \curK^{(m,N)}(\cdot,\bfXi,\curM_{[a]}(\bfz)) \, ; \,  \bfz \in \RR^\nu\} \, .
\end{equation}
It should be noted that, for $\bfz = \bfzero$, $[y] = \curM_{[a]}(\bfzero) =[a]$, which corresponds to the random field $[\bfK^{(m,N)}]  = \curK^{(m,N)}(\cdot,\bfXi,[a])$. The following proposition corresponds to Proposition~\ref{proposition3} for the approximation $[\bfK^{(m,N)}]$ of  random field $[\bfK]$.
\begin{proposition}\label{proposition4} 
The random field $[\bfK^{(m,N)}] = \curK^{(m,N)}(\cdot,\bfXi,\curM_{[a]}(\bfz))$, $\bfz \in\RR^{\nu}$, is such that
\begin{equation}
\Vert \bfK^{(m,N)}  \Vert_{F} \le \gamma(\bfXi,\bfz) < +\infty \label{boundKmN}
\end{equation}
almost surely and for all $\bfz\in\RR^{\nu}$, where $\gamma : \RR^{N_{g}} \times \RR^{\nu}\rightarrow \RR$ is a measurable positive function. For the square type representation of random fields,  there exists a constant $ \overline{\gamma} $, independent on $N$ and $\bfz$, such that
\begin{equation}
E\{\gamma(\bfXi,\bfz)\} \le  \overline{\gamma} <+\infty\quad  \text{for all }\bfz\in\RR^{\nu}. \label{boundgamma}
\end{equation}
Moreover, if $\sum_{i=1}^{+\infty} \sqrt{\sigma_i} \, \Vert G_i \Vert_\infty < +\infty$, with $\sigma_i$ and $[G_i]$ defined in Section~\ref{Section2.3.1}, \eqref{boundgamma}
is satisfied for a constant  $\overline{\gamma}$ independent of $m$.
\end{proposition}
\begin{proof}
Following the proof of Lemma \ref{lemma6}, we obtain
\begin{equation}
\Vert \bfG^{(m,N)}(\bfx) \Vert_{F} \le  \Vert G_{0} \Vert_{\infty} + \sum_{i=1}^{m} \sqrt{\sigma_{i}} \Vert G_{i}\Vert_{\infty} \vert \eta_{i}(\bfXi,\bfz) \vert:=\delta(\bfXi,\bfz),\label{eq:delta}\end{equation}
with $\boldsymbol{\eta} = (\eta_{1},\ldots,\eta_{m})= \curM_{[a]}(\bfz)^{T}\bfPsi(\bfXi)$.
Using Proposition~\ref{proposition3}, we then obtain Eq.~\eqref{boundKmN} with
$\gamma= \underline{k}_{1}\sqrt{n}(\varepsilon +e^{\delta})/(1+\varepsilon)$
for the exponential type representation and
$\gamma = \underline k_{\, 1} (\sqrt{n}\, \varepsilon + \gamma_0 +\gamma_1\,\delta^2)/(1+\varepsilon)$ for the square type representation. Using $E\{\eta_{\, i}^{2}\}=1$, it can be shown that $E\{\delta^2\} \le  2\Vert G_{0}\Vert_{\infty}^{2} + 2\left( \sum_{i=1}^{m} \sqrt{\sigma_{i}} \Vert G_{i}\Vert_{\infty}  \right)^{2} 
:= \overline \zeta_{m}$, where $\overline \zeta_{m}<+\infty$ is independent on $N$ and $\bfz$. Therefore, for the square type representation, we obtain Eq.~\eqref{boundgamma} with
$\overline{\gamma} = \underline k_{\, 1} (\sqrt{n}\, \varepsilon + \gamma_0 +\gamma_1\,\overline{\zeta}_{m})/(1+\varepsilon)$. Moreover, if $\sum_{i=1}^{+\infty} \sqrt{\sigma_i} \, \Vert G_i \Vert_\infty < +\infty$, then $ \overline \zeta_{m}\le  \overline \zeta_{\infty} <+\infty$ and  Eq.~\eqref{boundgamma} holds for $\overline{\gamma} = \underline k_{\, 1} (\sqrt{n}\, \varepsilon + \gamma_0 +\gamma_1\,\overline{\zeta}_{\infty})/(1+\varepsilon)$.
 \end{proof}

\subsection{Brief description of the identification procedure}
\label{Section3.3}
Let $\curB$ be the nonlinear mapping which, for any given random field $[\bfK] = \{[\bfK(\bfx)], \bfx\in D\}$ introduced and studied in Section~\ref{Section2}, associates a unique random observation vector $\bfU^\obs = \curB([\bfK])$ with values in $\RR^{m_\pobs}$. The nonlinear mapping $\curB$ is constructed in solving the elliptic stochastic boundary value problem as explained in Section~\ref{Section4}. We are then interested in identifying the random field $[\bfK]$ using partial and limited experimental data set $\bfu^{\exper , 1}, \ldots , \bfu^{\exper, \nu_\experp} $ in $\RR^{m_\pobs}$ ($m_\obs$ and $\nu_\exper$ are small).
In high stochastic dimension (which is the assumption of the present paper), such a statistical inverse problem is an  ill-posed problem if no additional available information is introduced. As explained in \cite{Soize2010,Soize2011}, this difficulty can be circumvented (1) in introducing an algebraic prior model (APM), $[\bfK^\APM]$, of random field $[\bfK]$, which contains additional information and satisfying the required mathematical properties and (2) in using an adapted identification procedure. Below, we  summarize and we adapt this procedure to the new representations and their parametrizations of random field $[\bfK]$ that we propose in this paper. The  steps of the identification procedure are the following:

\paragraph*{Step 1} Introduction of a family  $\{[\bfK^\APM(\bfx;\bfw)],$ $\bfx\in D\}$ of  prior algebraic models for random field $[\bfK]$. This family depends on an unknown parameter $\bfw$ (for instance, $\bfw$ can be made up of the mean function, spatial correlation lengths, dispersion parameters controlling the statistical fluctuations, parameters controlling the shape of the tensor-valued correlation function, parameters controlling the symmetry class, etc). For fixed $\bfw$, the probability law and the generator of independent realizations of the APM are known.
{For example, for the modeling of groundwater Darcy flows, a typical choice for the APM would consist in a homogeneous and isotropic lognormal random field $[\bfK^\APM(\bfx;\bfw)]=\exp(G^\APM) [I_d]$ with $G^\APM$ a homogeneous real-valued gaussian random field having a covariance function of the Mat\'ern family. With this APM, $\bfw$ consists of 4 scalar parameters that are the mean value of $G$ and the three parameters of the Mat\'ern covariance.}
For {other} examples of algebraic prior models of non-Gaussian positive-definite matrix-valued random fields, we refer the reader to \cite{Soize2006} for the anisotropic class, to \cite{Ta2010} for the isotropic class, to \cite{Guilleminot2011a} for bounded random fields in the anisotropic class, to \cite{Guilleminot2011b} for random fields with any symmetry class (isotropic, cubic, transversal isotropic, tetragonal, trigonal, orthotropic), and finally, to \cite{Guilleminot2013} for a very general class of bounded random fields with any symmetry properties from the isotropic class to the anisotropic class.
%
\paragraph*{Step 2} Identification of an optimal value $\bfw^\opt$ of parameter $\bfw$ using the experimental data set, the family of stochastic solutions
$\bfU^\obs(\bfw) = \curB([\bfK^\APM(\cdot;\bfw)])$  and a statistical inverse method such as the moment method, the least-square method or the maximum likelihood method \cite{Serfling1980,Spall2003,Walter1997,Soize2010}). The optimal algebraic prior model  $\{[\bfK^\OAPM(\bfx)], \bfx\in D\} := \{[\bfK^\APM(\bfx;\bfw^\opt)], \bfx\in D\}$ is then obtained. Using the generator of realizations of the APM, $\nu_\KL$ independent realizations $[K^{(1)}],\ldots , [K^{(\nu_\KL)}]$ of random field $[\bfK^\OAPM]$ can be generated with $\nu_\KL$ as large as it is desired without inducing a significant computational cost.
%
\paragraph*{Step 3} Choice of a type of representation for random field $[\bfK]$ and, using Eq.~\eqref{Eq2} with Eq.~\eqref{Eq9} or with Eq.~\eqref{Eq17bis}, the optimal algebraic prior model $\{[\bfG^\OAPM(\bfx)],$ $\bfx\in D\}$ of random field $[\bfG]$ is deduced.
For all $\bfx\in D$,  $[\bfG^\OAPM(\bfx)] = \log_\MM([\bfK_0^\OAPM(\bfx)])$ for the exponential type representation and  $[\bfG^\OAPM(\bfx)] = \LL^{-1}([\bfK_0^\OAPM(\bfx)])$ for the square type representation, with
\begin{equation}                                                                                                                                \label{Eq47}
[\bfK_0^\OAPM(\bfx)] =(1+\varepsilon) [\underline L(\bfx)]^{-T}\, [\bfK^\OAPM(\bfx)] \, [\underline L(\bfx)] -  \varepsilon [\, I_n] \, .
\end{equation}
It is assumed that random field $[\bfG^\OAPM]$ belongs to $\curC^{2,S}_n$.
From the $\nu_\KL$ independent realizations $[K^{(1)}],\ldots , [K^{(\nu_\KL)}]$ of random field $[\bfK^\OAPM]$, it can be deduce the $\nu_\KL$ independent realizations $[G^{(1)}],\ldots , [G^{(\nu_\KL)}]$ of random field $[\bfG^\OAPM]$.

\paragraph*{Step 4} Use of the $\nu_\KL$ independent realizations $[G^{(1)}],\ldots , [G^{(\nu_\KL)}]$ of random field $[\bfG^\OAPM]$ and use of the adapted statistical estimators for estimating the mean function  $[G_0^\OAPM]$ and the covariance function $C_{\bfG^\pOAPM}$ of random field $[\bfG^\OAPM]$ (see Section~\ref{Section2.3.1}). Then, calculation of  the first $m$ eigenvalues $\sigma_1\geq \ldots \geq \sigma_m$ and the corresponding eigenfunctions
$[G_1],\ldots, [G_m]$ of the covariance operator $\hbox{Cov}_{\bfG^\pOAPM}$ defined by the kernel $C_{\bfG^\pOAPM}$. For a given convergence tolerance with respect to  $m$, use of Eq.~\eqref{Eq36} to construct independent realizations of the random vector $\bfeta^\OAPM = (\eta^\OAPM_1, \ldots , \eta^\OAPM_m)$ such that
\begin{equation}                                                                                                                               \label{Eq48}
[\bfG^{\OAPM (m)}(\bfx)] = [G_0^\OAPM(\bfx)] + \sum_{i=1}^{m} \sqrt{\sigma_i}\, [G_i(\bfx)] \, \eta^\OAPM_i\, .
\end{equation}
For $i=1,\ldots , m$, the $\nu_\KL$ independent realizations $\eta_i^{(1)},\ldots , \eta_i^{(\nu_\KL)}$ of the random variable $\eta^\OAPM_i$ are calculated by
\begin{equation}                                                                                                                                \label{Eq49}
\eta_i^{(\ell)} =\frac{1}{\sqrt{\sigma_i}} \ll [G^{(\ell)}] - [G_0^\OAPM] \, , [G_i]\gg \quad , \quad \ell = 1,\ldots , \nu_\KL \, .
\end{equation}

\paragraph*{Step 5} For a given convergence tolerance with respect to $N$ and $N_g$ in the polynomial chaos expansion  defined by Eq.~\eqref{Eq38bis}, use of the methodology based on the maximum likelihood and the corresponding algorithms presented in \cite{Soize2010} for estimating a value $[y_0]\in\VV_m(\RR^N)$ of $[y]$ such that $\bfeta^\OAPM = [y_0]^T\,  \bfPsi(\bfXi)$.

Let us examine the following particular case for which the algebraic prior model $[\bfK^\APM]$ of random field $[\bfK]$ is defined by Eq.~\eqref{Eq2} with either Eq.~\eqref{Eq9} or Eq.~\eqref{Eq17bis}, in which $[\bfG^\APM]$ is chosen as a second-order Gaussian random field indexed by $D$ with values in $\MM_n^{\rm S}(\RR)$. Therefore,  the components $\eta_1, \ldots , \eta_m$ of the random vector $\bfeta$ defined by Eq.~\eqref{Eq38bis} are independent real-valued normalized Gaussian random variables. We then have  $N_g=m$. Let us assume that, for $1 \leq j \leq m \leq N$, the indices $j$ of the polynomial chaos are ordered such that  $\Psi_j(\bfXi) = \Xi_j$. It can then be deduced that $[y_0]\in\VV_m(\RR^N)$ is such that $[y_0]_{ji} = \delta_{ij}$ for $1\leq i\leq m$ and $1\leq j\leq N$.

\paragraph*{Step 6} With the maximum likelihood method, estimation of a value $\check{\bfz}$ of $\bfz\in\RR^\nu$ for the parameterized general class $\curK$ defined by Eqs.~\eqref{Eq40} and \eqref{Eq46}, using the family of stochastic solutions $\bfU^\obs(\bfz) = \curB(\curK^{(m,N)}(\cdot,\bfXi,\curM_{[y_0]}(\bfz)))$ and the experimental data set  $\bfu^{\exper , 1}, \ldots , \bfu^{\exper, \nu_\experp} $. Then, calculation of $[\check y] = \curM_{[y_0]}(\check\bfz)$.

\paragraph*{Step 7} Construction of a posterior model for random field $[\bfK]$ using the Bayesian method. In such a framework, the coefficients $[y]$ of the polynomial chaos expansion $\bfeta = [y]^T\,  \bfPsi(\bfXi)$ (see Eq.~\eqref{Eq38bis}) are modeled by a random matrix $[\bfY]$ (see \cite{Soize2009}) as proposed in \cite{Soize2011} and consequently, $\bfz$ is modeled by a $\RR^\nu$-valued random variable $\bfZ$. For the prior model $[\bfY^\prior]$ of $[\bfY]$, here we propose
\begin{equation}                                                                                                                                \label{Eq50}
[\bfY^\prior]  = \curM_{[{\check y}]}(\bfZ^\prior) \quad , \quad \bfZ^\prior \sim \hbox{centered Gaussian vector}\, ,
\end{equation}
which guaranties that $[\bfY^\prior]$ is a random matrix with values in $\VV_m(\RR^N)$
whose statistical fluctuations are centered around $[{\check y}]$ (the maximum likelihood estimator of the set of
coefficients $[y]$ computed at step 6).
The Bayesian update allows the posterior distribution of random vector $\bfZ^\post$ to be estimated using the stochastic solution $\bfU^\obs = \curB(\curK^{(m,N)}(\cdot,\bfXi,\curM_{[{\check y}]}(\bfZ^\prior)))$ and the experimental data set $\bfu^{\exper , 1}, \ldots , \bfu^{\exper, \nu_\experp} $.\\

Finally, it should be noted that once the probability distribution of $\bfZ^\post$ has been estimated by Step~7, $\nu_\KL$ independent realizations can be calculated for the random field $[\bfG^\post(\bfx)] =  [G_0^\OAPM(\bfx)] + \sum_{i=1}^{m} \sqrt{\sigma_i}\, [G_i(\bfx)] \, \eta^\post_i$
in which $\bfeta^\post = [\bfY^\post]^T\, \bfPsi(\bfXi)$ and where $[\bfY^\post] = \curM_{[{\check y}]}(\bfZ^\post)$. The identification procedure can then be restarted from Step~4 replacing $[\bfG^\OAPM]$ by $[\bfG^\post]$.

\section{Solution of the stochastic elliptic boundary value problem}
\label{Section4}

Let $D\subset \mathbb{R}^{d}$  be a bounded open domain with smooth boundary.
 The following elliptic
 stochastic partial differential equation is considered,
\begin{equation}
-{\rm div}([\bfK] \cdot \nabla U) = f\quad \text{a.e. in $D$}\, , \label{edps}
\end{equation}
with homogeneous Dirichlet boundary conditions (for the sake of simplicity). The random field $\{[\bfK(\bfx)], \bfx\in D\}$  belongs to the parameterized general class of random fields,
$$\curK=\{\curK^{(m,N)}(\cdot,\bfXi,\curM_{[a]}(\bfz)) \, ; \, \bfz\in \RR^{\nu} \}\, ,$$
introduced in Section \ref{Section3.2}.
This stochastic boundary value problem has to be solved at steps 6 and 7 of the identification
procedure described in Section \ref{Section3.3},  respectively considering $\bfz$ as a deterministic or a random
parameter.
We introduce the map $$u:D\times \mathbb{R}^{N_{g}} \times \mathbb{R}^{\nu} \rightarrow \RR $$
 such that $U = u(\bfx,\bfXi,\bfz)$ is the solution of the stochastic boundary value problem for  $[\bfK(\bfx)] = \curK^{(m,N)}(\bfx,\bfXi,\curM_{[a]}(\bfz))$.
 The aim is here to construct an explicit approximation  of the map $u$ for its efficient use in the
identification procedure.
\\

In this section,  a suitable functional framework  will first be introduced for the definition of the map $u$. Then, numerical methods based on Galerkin projections will be analyzed for the approximation of this map.
Different numerical approaches will be introduced depending on the type
of representation of random fields (square type or exponential type) and depending on the properties of approximation spaces.
For the two types of representation of random fields (exponential type or square type).
Finally, we will briefly describe complexity reduction methods based on low-rank approximations
that exploit the tensor structure of the high-dimensional map $u$ and allows its approximate representation to be obtained in high dimension. That makes affordable the application of the identification procedure for high-dimensional germs $\bfXi$ (high $N_{g}$) and high-order representation of random fields (high $\nu$).

\subsection{Analysis of the stochastic boundary value problem}

We denote by $\Gamma = \Gamma_{N_{g}} \otimes \Gamma_{\nu}$ a product measure on
 $\RR^{\mu} := \RR^{N_{g}} \times  \RR^{\nu}$, where  $\Gamma_{N_{g}}$ is the probability measure of random variable $\bfXi$ and where
 $\Gamma_{\nu}$ is a finite measure on $\RR^{\nu}$. Up to a normalization, $\Gamma_{\nu}$ is considered as the probability measure of a random vector $\bfZ$.
We denote by $[\Cdroit] : D\times \mathbb{R}^{N_{g}} \times \mathbb{R}^{\nu} \rightarrow \MM_n^+(\RR)$  the map defined by
$$
[\Cdroit](\bfx,\bfy,\bfz)  =  \curK^{(m,N)}(\bfx,\bfy,\curM_{[a]}(\bfz))\, ,
$$
and such that $[\bfK(\bfx)] = [\Cdroit](\bfx,\bfXi,\bfZ)$ is a $\sigma(\bfXi,\bfZ)$-measurable random field. Sometimes, the random field $\{[\Cdroit](\bfx,\bfXi,\bfZ),\bfx\in D\}$ will be denoted by $\{[\bfC(\bfx)],\bfx\in D\}$.
For a measurable function $h:\RR^{N_{g}} \times  \RR^{\nu}\rightarrow \RR$, the mathematical expectation of $h$ is defined by
 $$E_{\Gamma}(h) = E\{h(\bfXi,\bfZ)\}= \int_{\RR^{N_{g}}\times \RR^{\nu}} h(\bfy,\bfz) \, \Gamma(d\bfy,d\bfz)\, .$$
%
\begin{lemma}\label{lemma7} 
 Under the hypotheses of Proposition \ref{proposition4},
 there exists a constant $\alpha$  and a positive measurable function $\gamma : \RR^{N_{g}} \times  \RR^{\nu} \rightarrow \RR$ such that, for $\Gamma$-almost all $(\bfy,\bfz)$ in $\RR^{N_{g}} \times  \RR^{\nu}$, we have
\begin{equation}
0<\alpha \le \mathop{\rm ess \; inf}\limits_{\bfx \in D} \inf_{\bfh\in \RR^{n}\setminus \{0\}} \frac{<[\Cdroit](\bfx,\bfy,\bfz)\bfh,\bfh>_{2}}{\Vert \bfh \Vert_{2}^{2}} \, ,
\end{equation}
\begin{equation}
 \mathop{\rm ess \; sup}\limits_{\bfx \in D} \sup_{\bfh\in \RR^{n}\setminus \{0\}} \frac{<[\Cdroit](\bfx,\bfy,\bfz)\bfh,\bfh>_{2}}{\Vert \bfh \Vert_{2}^{2}} \le \gamma (\bfy,\bfz)<\infty \, .
\end{equation}
Moreover, for the class of random fields corresponding to the square type representation, we have
%
$E_{\Gamma}(\gamma)\le \overline \gamma <+\infty,$ 
%
that means $\gamma\in L^{1}_{\Gamma}(\RR^{\mu})$.
\end{lemma}
\begin{proof}
Lemma \ref{lemma1}(ii) gives the existence of the lower bound $\alpha = \underline{k}_{\varepsilon} $.  Proposition~\ref{proposition4} yields $\Vert [\Cdroit](\bfx,\bfy,\bfz)\Vert_{2} \le \Vert [\Cdroit] (\bfx,\bfy,\bfz)\Vert_{F} \le \gamma(\bfy,\bfz)<\infty$,
 with $\gamma$ a measurable function defined  on $\RR^{N_{g}}\times \RR^{\nu}$. For the square type representation of random fields, property \eqref{boundgamma} implies $E_{\Gamma}(\gamma) = \int_{\RR^{\nu}} E_{\Gamma_{N_g}}\{\gamma(\bfXi,\bfz) \}\Gamma_{\nu}(d\bfz) \le \overline{\gamma}$.
 \end{proof}
%
%
We introduce the bilinear form
 $C(\cdot,\cdot;\bfy,\bfz):H_{0}^{1}(D)\times H_{0}^{1}(D)\rightarrow \RR$ defined by
\begin{equation}
C(u,v;\bfy,\bfz)  = \int_{D} \nabla v \cdot [\Cdroit](\cdot,\bfy,\bfz)  \nabla u \, d\bfx \, .
\end{equation}
Let introduce $\Vert \cdot\Vert_{H^{1}_{0}}=\left(\int_{D} \vert\nabla (\cdot)\vert^{2}\, d\bfx\right)^{1/2}$ the norm on $H^{1}_{0}(D)$, and
$\Vert\cdot\Vert_{H^{-1}}$ the norm on the continuous dual space $H^{-1}(D)$.
\paragraph{Strong-stochastic solution}
\ 
\begin{proposition}\label{proposition5}  
Assume $f \in H^{-1}(D)$. Then, for $\Gamma$-almost all $(\bfy,\bfz)$ in $\RR^{N_{g}}\times\RR^{\nu}$, there exists a unique
$u(\cdot,\bfy,\bfz)\in H_{0}^{1}(D)$ such that
\begin{equation}
C(u(\cdot,\bfy,\bfz),v;\bfy,\bfz)  = f(v) \quad \text{for all } v\in H_{0}^{1}(D)\, , \label{eq:strongstochasticform}
\end{equation}
and
\begin{equation}
\Vert u(\cdot,\bfy,\bfz)   \Vert_{H^{1}_{0}}\le \frac{1}{\alpha} \Vert f\Vert_{H^{-1}}\, . \label{eq:continuity_almostsure}
\end{equation}
\end{proposition}
\begin{proof}
Lemma \ref{lemma7} ensures the continuity and coercivity of bilinear form $C(\cdot,\cdot;\bfy,\bfz)$ for $\Gamma$-almost all $(\bfy,\bfz)$ in $\RR^{N_{g}}\times\RR^{\nu}$. The proof follows from a direct application of the Lax-Milgram theorem.
\end{proof}

\paragraph{Weak-stochastic solution}
Let  be $L^{2}_{\Gamma}(\RR^{\mu})=L^{2}_{\Gamma_{N_{g}}}(\RR^{N_{g}}) \otimes L^{2}_{\Gamma_{\nu}}(\RR^{\nu}) $ and
$X = H_{0}^{1}(D) \otimes L^{2}_{\Gamma}(\RR^{\mu})$. Then $X$ is a Hilbert space for the inner
 product norm $\Vert\cdot\Vert_{X}$ defined by
$$ \Vert v\Vert_X^2= {E}_{\Gamma}\big(\int_D \Vert\nabla v \Vert_{2}^2 \, d\bfx \big). $$
We also introduce the spaces $X^{(\gamma^{s}{\bfC}^{r})}$ ($s,r\in \NN$)
of functions $v:D\times \RR^{\mu}\rightarrow\RR$ with bounded norm
$$
 \Vert v\Vert_{X^{(\gamma^{s}{\bfC}^{r})}}=\Big\{{E}_{\Gamma}\big( \gamma^{s} \int_D
 \nabla v \cdot [\Cdroit]^{r} \nabla v \, d\bfx\big)\Big\}^{{1/2}}.
$$
\begin{lemma}\label{lemma8}
We have
$\alpha^{{3/2}} \Vert v\Vert_{X} \le  \alpha \Vert v\Vert_{X^{(\bfC)}} \le \Vert v\Vert_{X^{({\bfC}^2)}} \le \Vert v\Vert_{X^{(\gamma\bfC)}}\le \Vert v\Vert_{X^{(\gamma^2)}},$ and therefore
$$X^{(\gamma^2)}\subset X^{(\gamma\bfC)} \subset X^{({\bfC}^2)}\subset X^{(\bfC)} \subset X$$
with dense embeddings.
\end{lemma}
\begin{proof}
The inequalities satisfied by the norms are easily deduced from the properties of $[\Cdroit]$
(Lemma \ref{lemma7}). Then, it can easily be proven that $X^{(\gamma^2)}$ is dense
in $X$, which proves the density of other embeddings.
\end{proof}
Let introduce the bilinear form $a:X\times X\rightarrow
\mathbb{R}$ defined by
$$
a(u,v) = E_{\Gamma}(C(u,v)) = \int_{\RR^\mu} \left(
 \int_{D} \nabla v \cdot [\Cdroit](\cdot,\bfy,\bfz)  \nabla u \, d\bfx\right) \Gamma(d\bfy,d\bfz)\, ,
$$
and the linear form $F$ belonging to the continuous dual space $X'$ of $X$, defined by
$$
\left<F,v\right> = E_{\Gamma}(f(v))=\int_{\RR^\mu} f(v(\cdot,\bfy,\bfz)) \Gamma(d\bfy,d\bfz)\, .
$$
From Lemma \ref{lemma7}, it can easily be deduced  the
\begin{lemma}\label{lemma9}
$a:X\times X \rightarrow \mathbb{R}$ is a symmetric bilinear form such that:
\begin{enumerate}[(ii)]\renewcommand{\theenumi}{\roman{enumi}}

\item \label{continuity}
$a$ is continuous from
$X^{({\bfC}^{2})}\times X$ to $\RR$,
\begin{equation}\vert a(u,v)\vert \le \Vert u\Vert_{X^{({\bfC}^{2})}}\Vert v\Vert_{X} \quad \forall (u,v)\in
X^{({\bfC}^{2})}\times X \, .\label{eq:continuity}
\end{equation}
\item $a$ is continuous from
$X^{(\bfC)}\times X^{(\bfC)}$ to $\RR$,
\begin{equation}\vert a(u,v)\vert \le \Vert u\Vert_{X^{(\bfC)}}\Vert v\Vert_{X^{(\bfC)}} \quad \forall (u,v)\in X^{(\bfC)}\times X^{(\bfC)}\, . \label{eq:continuityC}
\end{equation} \item \label{coercivity}$a$
is coercive,
\begin{equation} a(v,v)  \ge \alpha \Vert v \Vert_X^2 \quad \forall v\in X\, .\label{eq:coercivity}
\end{equation}
\end{enumerate}
\end{lemma}
Now we introduce a weak form of the parameterized stochastic boundary value
problem:
\begin{equation}
\begin{split}
&\text{Find $u\in X$ such that}\\
&a(u,v) = F(v) \quad \forall v\in X \, .\end{split}\label{eq:prob_var}
\end{equation}
We have the following result.
\begin{proposition}\label{prop:well-posed}
There exists a unique solution $u\in X$ to problem \eqref{eq:prob_var}, and
%
$\Vert u\Vert_X \le \frac{1}{\alpha} \Vert F
\Vert_{X'}$.
%
Moreover, $u$ verifies \eqref{eq:strongstochasticform} $\Gamma$-almost surely.
\end{proposition}
\begin{proof}
The existence and uniqueness of  solution can be deduced from a general result obtained in \cite{Mugler2011}.
Here we provide a short  proof for completeness sake. We denote by $A : D(A)\subset X \rightarrow X'$ the linear operator
defined by $Au = a(u,\cdot)$, $D(A)$ being the domain of $A$. We
introduce the adjoint operator $A^*: D(A^*)\subset X \rightarrow X'$
defined by $\left<u,A^*v\right> = \left<Au,v\right>$ for all $u\in
D(A)$ and $v\in D(A^*)$, with $D(A^*)=\{v\in X; \exists c>0\text{
such that } |\left<Au,v\right>|\le c \Vert u\Vert_X\text{ for all }
u\in D(A)\}$. The continuity property \eqref{eq:continuity} implies that $A$
is continuous from $X^{({\bfC}^2)}$ to $X'$ and therefore,
$D(A)\supset X^{({\bfC}^{2})}$ is dense in $X$ (using Lemma \ref{lemma8}). That means that $A$ is
densely defined.
The coercivity property \eqref{eq:coercivity} implies that $
\Vert A v\Vert_{X'} \ge \alpha \Vert v\Vert_X $
for all $v\in X$, which implies that $A$ is injective and the range $R(A)$ of $A$ is
closed. It also implies that $A^*$ is injective and $R(A^*)$ is closed,
and by the Banach closed range theorem, we then have that $A$ is
surjective, which proves the existence of a unique solution. Then, 
using the coercivity of $a$, we simply obtain 
$
\Vert u\Vert_X^2 \le \frac{1}{\alpha} a(u,u) = \frac{1}{\alpha}
F(u) \le \frac{1}{\alpha} \Vert F\Vert_{X'}\Vert u\Vert_X.
$
\end{proof}

\begin{proposition}\label{proposition7}
The solution $u$ of problem \eqref{eq:prob_var} is such that:
\begin{enumerate}[(iii)]\renewcommand{\theenumi}{\roman{enumi}}
\item  $u\in X^{({\bfC})}$,
\item $u\in X^{(\gamma\bfC)}$ if $\gamma\in L^{1}_{\Gamma}(\RR^{\mu})$,
\item $u\in X^{(\gamma^2)}$ if $\gamma\in L^{2}_{\Gamma}(\RR^{\mu})$.     
\end{enumerate}
\end{proposition}
\begin{proof}
Using Eq.~\eqref{eq:strongstochasticform} with $v=u$ and Eq.~\eqref{eq:continuity_almostsure}, we obtain
$
\int_{D}\nabla u \cdot [\Cdroit] \nabla u \, d\bfx = f(u) \le \Vert f\Vert_{H^{-1}} \Vert u\Vert_{H^{1}_{0}}
 \le  \frac{1}{\alpha}\Vert f\Vert_{H^{-1}}^{2}.$ Taking the expectation yields $\Vert u\Vert_{X^{(\bfC)}}^{2}\le \frac{1}{\alpha}\Vert f\Vert_{H^{-1}}^{2}<\infty$, which proves (i).  If $\gamma\in L^{1}_{\Gamma}$,
$
\Vert u\Vert_{X^{(\gamma \bfC)}}^{2} = E_{\Gamma}(\gamma \int_{D}\nabla u \cdot [\Cdroit] \nabla u \, d\bfx) \le   \frac{1}{\alpha}\Vert f\Vert_{H^{-1}}^{2} E_{\Gamma}(\gamma)<\infty,$
which proves (ii). Finally,
if $\gamma\in L^{2}_{\Gamma}$, we have
$
\Vert u\Vert_{X^{(\gamma^{2})}}^{2} = E_{\Gamma}( \gamma^{2} \Vert u\Vert_{H^{1}_{0}}^{2} )  \le
\frac{1}{\alpha^{2}} \Vert f\Vert_{H^{-1}}^{2} E_{\Gamma}(\gamma^{2})<\infty,
$
which proves (iii).
\end{proof}

\subsection{Galerkin approximation}

Galerkin methods are introduced for the approximation of the
solution $u$ of problem \eqref{eq:prob_var}. Let $X_N \subset X$ be
an approximation space such that
$X_N:=X_{q,p}=V_{q}\otimes W_p$ with $V_q\subset H^{1}_{0}(D)$ (e.g. a finite
element approximation space) and $W_p\subset
L^2_\Gamma(\mathbb{R}^\mu)$ (e.g. a polynomial chaos approximation space). The Galerkin approximation $u_N\in X_N$ of $u$ is defined by
\begin{equation}
a(u_N,v_N) =  F(v_N)\quad \forall v_N\in X_N \, . \label{eq:prob_galerkin}
\end{equation}
Note that the coercivity property of bilinear form $a$ on
$X_N\times X_N$ (Lemma \ref{lemma9}(iii)) ensures the existence and uniqueness of a solution
$u_N$ to problem \eqref{eq:prob_galerkin}. The convergence of Galerkin
approximations is now analyzed in different situations corresponding to the different types
of representation of random fields (exponential or square type), to different underlying measures $\Gamma$, and to different choices of approximation spaces. We first analyze the case where  $X_{N}\subset X^{(\gamma)}$ which
leads us to a natural strategy for the definition of a convergent sequence of approximations. Then, we analyze the use of more general approximation spaces
that do not necessarily verify $X_{N}\subset X^{(\gamma)}$.

\subsubsection{Case $X_{N}\subset X^{(\gamma)}$}

\begin{proposition}\label{proposition8}
Assuming $X_{N}\subset X^{(\gamma)}\subset X^{(\bfC)}$, the solution $u_N\in X_N$ of
\eqref{eq:prob_galerkin} verifies
\begin{align}
&\Vert u-u_N\Vert_{X^{(\bfC)}}
\le \inf_{v_N\in X_N}\Vert u-v_N
\Vert_{X^{(\bfC)}} ,\label{eq:galerkin-XCnorm}
\\\text{and}\quad
&\Vert u-u_N\Vert_{X} \le \frac{1}{\sqrt{\alpha}}
 \inf_{v_N\in X_N}\Vert u-v_N
\Vert_{X^{(\bfC)}} .\label{eq:galerkin-Xnorm}
\end{align}
Therefore, if $\{X_{N}\}_{N\in\NN}\subset X^{(\gamma)}$ is a sequence of
approximation spaces such that
$\cup_{N\in\NN} X_{N}$ is dense in $X$, then there exists a subsequence of Galerkin approximations
$\{u_{N}\}_{N\in\NN}$ which converges to $u$ in the $X^{(\bfC)}$-norm and in the  $X$-norm.
\end{proposition}
\begin{proof}
Using the Galerkin
orthogonality property of $u_N$ and the continuity of $a$ (Lemma \ref{lemma9}(ii)) yield
\begin{eqnarray*}
 \Vert u-u_{N}\Vert_{X^{(\bfC)}}^{2}
&= a(u-u_{N},u-u_{N})
= a(u-u_N,u-v_N) 
\le \Vert u-u_N\Vert_{X^{(\bfC)}}\Vert
u-v_N\Vert_{X^{(\bfC)}},\end{eqnarray*}
for all $v_N\in X_N$. Inequality \eqref{eq:galerkin-XCnorm} is obtained
by taking the infimum over all $v_N\in X_N$. Then, inequality
 \eqref{eq:galerkin-Xnorm} is obtained by using the coercivity of $a$ (Lemma \ref{lemma9}(iii)).
\end{proof}

Let us now analyze the condition $X_N\subset X^{(\gamma)}$ of Proposition \ref{proposition8}, with $X_{N}=V_{q}\otimes W_{p}$. Due to
the tensor product structure of $X_{N}=V_{q}\otimes W_{p}$, the  condition $X_{N}\subset X^{(\gamma)}$ is equivalent to
$ \sqrt{\gamma}\, W_p \subset L^2_{\Gamma}(\mathbb{R}^\mu)$, that means
$
{E}_{\Gamma}(\varphi^2\gamma)<+\infty  \text{ for all }
\varphi\in W_p.
$

\paragraph{Use of weighted approximation spaces (for both types of representation of random fields)}
A possible approximation strategy consists in choosing weighted approximation spaces $W_{p} = \{\gamma^{-1/2} \varphi ;\varphi\in \widetilde W_{p}\}$ with $\widetilde W_{p}\subset L^{2}_{\Gamma}(\RR^{\nu})$. For example, if the measure $\Gamma$ has finite moments of any order, $\widetilde W_{p}$ can be chosen as a classical (piecewise) polynomial space with degree $p$, e.g.
 $$
 {\widetilde W}_p=  \mathbb{P}_p(\mathbb{R}^{N_{g}}) \otimes \mathbb{P}_{p}(\RR^{\nu})=
\mathrm{span}\{\bfy^\bfalpha\bfz^{\bfbeta};\bfalpha\in\NN^{N_{g}},\bfbeta\in\NN^{\nu} ,  \vert \bfalpha\vert \le
p,\vert \bfbeta\vert\le p\}.
$$
Indeed, for all $\varphi = \gamma^{-1/2} \bfy^{\bfalpha}\bfz^{\bfbeta}\in W_{p}$,$$
{E}_{\Gamma}(\varphi^{2}\gamma) = {E}_{\Gamma}(\bfy^{2\bfalpha} \bfz^{2\bfbeta}) = \int_{\mathbb{R}^\mu}
\bfy^{2\bfalpha} \bfz^{2\bfbeta} \Gamma(d\bfy,d\bfz) <+\infty,
$$
This approximation strategy is adapted to both types of representation of random fields (exponential or square type) since it does not require any assumption on $\gamma$.
However, the use of weighted  polynomial approximation spaces $W_{p}$ leads to non classical computational treatments.

\paragraph{Use of classical approximation spaces (for square type representation of random fields)} More classical approximation spaces can be used provided some conditions on $\gamma$.
\begin{lemma}\label{lemma10}
If $\gamma \in L^r_\Gamma(\mathbb{R}^\mu)$ and $W_p\subset
L^s_\Gamma(\mathbb{R}^\mu)$ with some $r\ge 1$ and $s\ge 2$ such that
$\frac{2}{s}+\frac{1}{r}=1$, then $X_N \subset X^{(\gamma)}$. In particular $X_{N}\subset X^{(\gamma)}$ if $\gamma \in L^{1}_{\Gamma}(\RR^{\mu})$ and
 $W_{p}\subset L^{\infty}_{\Gamma}(\RR^{\mu})$.
\end{lemma}
\begin{proof}
For $r>1$, letting $s'=\frac{s}{2}$ such that $\frac{1}{s'}+\frac{1}{r}=1$, we have
$$ {E}_{\Gamma}(\varphi^2\gamma) \le
\{{E}_{\Gamma}(\varphi^{2s'})\}^{1/s'}\, \{{E}_{\Gamma}(\gamma^{r})\}^{1/r} =
\{{E}_{\Gamma}(\varphi^{s})\}^{2/s}\, \{{E}_{\Gamma}(\gamma^{r})\}^{1/r}<\infty.
$$
The proof for $r=1$ and $s=\infty$ is straightforward.
\end{proof}

Lemma \ref{lemma10}
 allows us to analyze the approximation when using the square type representation of random fields, for which
$\gamma\in L^{1}_{\Gamma}(\RR^{\mu})$. Indeed, in this case, Lemma \ref{lemma10}
 implies that $X_{N}\subset X^{(\gamma)}$ if
 $W_{p}\subset L^{\infty}_{\Gamma}(\RR^{\mu})$. In particular, this condition is satisfied if $\Gamma$ has a bounded support and $W_{p}$ is a (possibly piecewise) polynomial space, e.g. $W_{p}=\PP_{p}(\RR^{N_{g}})\otimes \PP_{p}(\RR^{\nu})$.
We note that $\Gamma = \Gamma_{N_{g}}\otimes \Gamma_{\nu}$ has a bounded support if (i) the random germ
$\bfXi$ for the representation of random fields in the class $\curK=\{\curK^{(m,N)}(\cdot,\bfXi,\curM_{[a]}(\bfz)) \, ; \, \bfz\in \RR^{\nu} \}$ has a bounded support $supp(\Gamma_{N_{g}})$, and (ii) the support of the measure $\Gamma_{\nu}$ on the parameter space  $\RR^{\nu}$ is bounded. This latter condition implies that the map $u$ only provides the solution to the boundary value problems associated to a subset of random fields in $\curK$ (those corresponding to parameters $\bfz\in supp(\Gamma_{\nu}))$. In other words, this first result shows that  when using the square type representation of random fields with a germ $\bfXi$ with bounded support, the approximation is possible using classical polynomial approximation spaces.  
\\

Note that for having $X_{N}\subset X^{(\gamma)}$, weaker conditions on approximation spaces  could be obtained by looking further at the properties of $\gamma$ and the measure $\Gamma$.
In particular, if $W_p$ is the tensor product
of polynomial spaces with (total or partial) degree $p$, that means
$W_p=\mathbb{P}_p(\mathbb{R}^{N_{g}})\otimes \PP_{p}(\RR^{\nu})$, then $X_N\subset X^{(\gamma)}$ if for all $\bfy^{\bfalpha}\bfz^{\bfbeta}\in W_{p}$,
\begin{equation}
{E}_{\Gamma}(\bfy^{2\bfalpha} \bfz^{2\bfbeta}\gamma(\bfy,\bfz)) = \int_{\mathbb{R}^\mu}
\bfy^{2\bfalpha} \bfz^{2\bfbeta}\gamma(\bfy,\bfz) \Gamma(d\bfy,d\bfz) <+\infty,\label{eq:conditionWpgamma}
\end{equation}
which is a condition on $\gamma$ and on measure $\Gamma$. The following result justifies the applicability of classical (piecewise) polynomial spaces when using the  square type representation of random fields.

\begin{proposition}\label{proposition9}
For the square type representation of random fields, if $W_{p}$ is the tensor product
of (possibly piecewise) polynomial spaces and if $\Gamma=\Gamma_{N_{g}}\otimes \Gamma_{\nu}$ is such that probability measures
$\Gamma_{N_{g}}$ and $\Gamma_{\nu}$ admit moments of any order, then $X_{N} =V_{q}\otimes W_{p} \subset X^{(\gamma)}$.
\end{proposition}
\begin{proof}
Following the proof of Proposition  \ref{proposition4}, we obtain
$\gamma \le \underline{k}_{1}(\sqrt{n} \varepsilon +\gamma_{0} +\gamma_{1}\delta^{2})$, with $\delta$ defined by Eq.~\eqref{eq:delta} and such that
\begin{eqnarray*}
\delta^{2} \le  2\Vert G_{0}\Vert_{\infty}^{2} + 2\Big(\sum_{i=1}^{m }\sqrt{\sigma_{i}}\Vert G_{i}\Vert_{\infty}\vert  \eta_{i}\vert\Big)^{2}
\le  g_{0}+ g_{1} \sum_{i=1}^{m} \eta_{i}^{2} ,
\end{eqnarray*}
with $g_{0}=2\Vert G_{0}\Vert_{\infty}^{2} $ and $g_{1}=  2 \sum_{i=1}^{m}\sigma_{i}\Vert G_{i}\Vert_{\infty}^{2} $.  From Eq.~\eqref{Eq38bis}, we have $\bfeta = [y]^{T}\bfPsi(\bfy)$ where $\bfPsi(\bfy)$ is a vector of polynomials in $\bfy$ and $[y]^{T}[y]=[I_{m}]$. The latter condition implies $\vert y_{i}^{j}\vert \le 1$ for all $1\le i \le m $ and $1\le j \le N $. Therefore, $\eta_{i}^{2} \le (\sum_{j=1}^{N} y^{j}_{i} \Psi_{j}(\bfy))^{2} \le N \sum_{j=1}^{N}\Psi_{j}(\bfy)^{2} $, and
$
\delta(\bfy,\bfz)^{2} \le g_{0}+ g_{1} m N \sum_{j=1}^{N}\Psi_{j}(\bfy)^{2} := Q(\bfy),$ where $Q(\bfy)$ is a polynomial in $\bfy$. If the measures $\Gamma_{N_{g}}$ and $\Gamma_{\nu}$ have finite moments of any order, then any (piecewise) polynomial function on $\RR^{\mu}$ is in $L^{1}_{\Gamma}(\RR^{\mu})$. Therefore,
if  $W_p$ is a space of (piecewise) polynomial functions on $ \RR^{\mu}$ , then for all $\varphi \in W_{p}$, we have
$$
{E}_{\Gamma}(\varphi(\bfy,\bfz)^{2}\gamma(\bfy,\bfz)) \lesssim {E}_{\Gamma}(\varphi(\bfy,\bfz)^{2}(1+Q(\bfy)) < +\infty,
$$
from which we deduce that $V_{q}\otimes W_{p} \subset X^{(\gamma)}$.
\end{proof}

Since measures with bounded support have finite moments of any order, Proposition \ref{proposition9} is consistent with the first conclusions of Lemma \ref{lemma10}. Moreover, we have that for the square type representation of random fields, if $\bfZ$ is chosen as a Gaussian random variable (that means $\Gamma_{\nu}$ is a Gaussian measure) or a random variable with bounded support, then the classical Hermite polynomial chaos space associated with a Gaussian germ $\bfXi$ can be used. Note that the use of a measure $\Gamma_{\nu}$ whose support is $\RR^{\nu}$ allows to explore the whole class of random fields $\curK$ with the single map $u$.

\begin{rem} The case of log-normal random fields corresponds to a particular case of the exponential type representation for which the random field $[\bfG]$ is Gaussian and represented using a degree one Hermite polynomial chaos expansion with a Gaussian germ $\bfXi$. In this case, it can be proven that  if $W_{p}$ is the tensor product
of (possibly piecewise) polynomial spaces and if $\Gamma_{\nu}$ admit moments of any order, then $X_{N} =V_{q}\otimes W_{p} \subset X^{(\gamma)}$. This justifies the use of polynomial approximation spaces when considering the particular case of log-normal random fields. However, this result does not extend to other random fields with exponential type representation. 
\end{rem}

\subsubsection{Case of general approximation spaces $X_{N}$}

In order to handle both types of representation of random fields (exponential type and square type) with a measure $\Gamma$ with a possibly unbounded support, we here propose and analyze an approximation strategy which consists in using truncated approximation spaces. We consider a family
of approximation spaces $\{V_q\}_{q\ge 1}$ such that $\cup_q V_q$ is
dense in $H^{1}_{0}(D)$, and a family $\{W_p\}_{p\ge 1}$ such that $\cup_p
W_p$ is dense in $L^2_{\Gamma}(\mathbb{R}^\mu)$. Classical (possibly piecewise) polynomial spaces with degree $p$ can be chosen for $W_{p}$. Let $X_{N}$ be defined as $V_{q}\otimes W_{p}$.
Then, for $\tau>0$, we introduce the approximation
space
$$W_p^\tau = \left\{ I_{\gamma\le \tau} \varphi;\varphi\in W_p
\right\}\, ,$$
where $I_{\gamma\le \tau}(\bfy,\bfz) = 1$ if $\gamma(\bfy,\bfz)\le \tau$ and $0$ if $\gamma(\bfy,\bfz)> \tau$.
It can be proven that $\cup_{\tau>0} \cup_{p\ge 1} W_{p}^{\tau}$
is dense in $L^2_{\Gamma}(\mathbb{R}^\mu)$. Let $X_{N}^\tau$ be defined as $V_q\otimes W_p^\tau$.
For $\psi = I_{\gamma\le
\tau}\varphi$ in $W_p^\tau$, we have
$
{E}_{\Gamma}({\gamma\psi^2}) \le \tau
{E}_{\Gamma}({\varphi^2})<\infty,
$
that means $X_{N}^\tau :=V_q\otimes W_p^\tau \subset
X^{(\gamma)}$ for all $\tau>0$.  The
Galerkin approximation of $u$ in $X_{N}^\tau$ is denoted by $u_{N}^\tau$.
\begin{proposition}\label{proposition10}
The Galerkin approximation $u_{N}^{\tau} \in X_{N}^{\tau}$ of $u$ satisfies
\begin{equation}
\Vert u-u_{N}^{\tau}\Vert_{X^{(\bfC)}}^{2} \le \tau \inf_{v\in X_{N}}\Vert u-v
\Vert_{X}^2 + \Vert u I_{\gamma>\tau}
\Vert_{X^{(\bfC)}}^2,
\end{equation}
and there exists a sequence of approximation spaces $X_{N(\tau)}^{\tau} = V_{q(\tau)}\otimes W_{p(\tau)}^{\tau}$ such that
\begin{equation}
\Vert u-u_{N(\tau)}^\tau\Vert_{X^{(\bfC)}} \to 0\quad \text{as}
\quad \tau\to\infty.\label{proposition10-convergence}
\end{equation}
\end{proposition}
\begin{proof}
We have $u\in X^{(\bfC)}$ (Proposition \ref{proposition7}) and  $X^{\tau}_{N}\subset X^{(\gamma)} \subset X^{(\bfC)}$.  We first note that
for all $v\in X_N^{\tau}$, 
\begin{align*}
\Vert (u-v)\Vert_{X^{(\bfC)}}^2 &= \Vert (u-v) I_{\gamma\le \tau}
\Vert_{X^{(\bfC)}}^2+ \Vert (u-v) I_{\gamma > \tau}
\Vert_{X^{(\bfC)}}^2\\
&\le \Vert (u-v) I_{\gamma\le \tau} \Vert_{X^{(\gamma)}}^2 + \Vert u
I_{\gamma > \tau} \Vert_{X^{(\bfC)}}^2\\
&\le \tau\Vert (u-v) I_{\gamma\le \tau} \Vert_{X}^2 + \Vert u
I_{\gamma > \tau} \Vert_{X^{(\bfC)}}^2.
\end{align*}
Then, using Proposition \ref{proposition8} (Eq.~\eqref{eq:galerkin-Xnorm}), we obtain
\begin{align*}
\Vert u-u_{N}^\tau\Vert_{X^{(\bfC)}}^2 &\le \inf_{v\in
X_{N}^\tau}\Vert u-v \Vert_{X^{(\bfC)}}^2 \le \tau \inf_{v\in X_{N}^{\tau}}\Vert (u-v)
I_{\gamma\le \tau} \Vert_{X}^2 +  \Vert u
I_{\gamma>\tau} \Vert_{X^{(\bfC)}}^2 \\
&=  \tau\inf_{v\in X_{N}}\Vert (u-v)
I_{\gamma\le \tau} \Vert_{X}^2 + \Vert u
I_{\gamma>\tau} \Vert_{X^{(\bfC)}}^2\\
&\le\tau \inf_{v\in X_{N}}\Vert u-v
\Vert_{X}^2 + \Vert u I_{\gamma>\tau}
\Vert_{X^{(\bfC)}}^2
\end{align*}
Since $\Vert u \Vert_{X^{(\bfC)}}$ is bounded,
 the second term converges to $0$ as $\tau\to \infty$. Also,
provided that $\cup_{N} X_{N}$ is dense in $X$, we can define
a sequence of spaces $X_{N(\tau)}$ such that
$$
{\tau}\inf_{v\in X_{N(\tau)}}\Vert u-v \Vert_{X}^2
\to 0\quad \text{as} \quad \tau\to\infty,
$$
which ends the proof.
\end{proof}

\subsection{High-dimensional approximation using {sparse or} low-rank approximations}

The approximation of the  map  $ u \in X$ 
requires the introduction of adapted complexity reduction techniques. Indeed, when introducing approximation spaces
$V_{q}\subset H^{1}_{0}(D)$ and $W_{p} = W_{p}^{\bfy}\otimes W_{p}^{\bfz}\subset L^{2}_{\Gamma_{N_{g}}}(\RR^{N_{g}})\otimes L^{2}_{\Gamma_{\nu}}(\RR^{\nu})$, the resulting approximation space $X_{N} = V_{q}\otimes W_{p}$  may have  a very high dimension $dim(V_{q})\times dim(W_{p})$. For classical non adapted constructions of approximation spaces $W_{p}$, the dimension of $W_{p}$ has  typically
an exponential (or factorial) increase with $N_{g}$ and $\nu$ (e.g. with polynomial spaces $W_{p}^{\bfy} = \PP_{p}(\RR^{N_{g}})$ and $W_{p}^{\bfz}=\PP_{p}(\RR^{\nu})$). Therefore, the use of standard approximation techniques would restrict the applicability of the identification procedure to a class of random fields
with a germ $\bfXi$ of small dimension (small $N_{g}$) and a low order in the representation of random fields (small~$\nu$).

{
\paragraph{Sparse tensor approximation}
A first possible way to reduce the complexity and address high-dimensional problems is to use adaptive sparse tensor approximation methods \cite{Cohen2010}. Supposing that $\{\phi_i = \otimes_{k=1}^\mu \phi_{i_k}^{(k)} ;i\in \mathbb{N}^\mu\}$ constitutes a basis of $W$, it consists in constructing a sequence of approximations 
$u_M   \in V_q \otimes W_M$, with  $W_M=\mathrm{span}\{\phi_i ;i\in \Lambda_M\}$, where $\Lambda_M \subset \mathbb{N}^\mu$ is a subset of $M$ indices constructed adaptively. For some classes of random fields (that are simpler than the one proposed in the present paper), algorithms have been proposed for the adaptive construction of the sequence of subsets $\Lambda_M$ and some results have been obtained for the convergence of the corresponding sequence of approximations. In \cite{Cohen2013}, convergence results have been obtained for different classes of random fields and in particular for two classes of random fields that are closely related to the ones proposed in the present paper, namely random fields that are obtained with the exponential or the square of a random field which admits an affine decomposition in terms of the parameters. Under suitable conditions on the convergence of the series expansion of this random field, the authors prove the convergence with a convergence rate independent on the dimension $\mu$. Such results may be obtained for the class of random fields proposed in the present paper under some conditions on the underlying second order random field $[\bfG]$. This will be addressed in future work. 
}

\paragraph{Low-rank approximations.}
Low-rank tensor approximation methods can be used in order to reduce the complexity of the approximation of functions in tensor product spaces \cite{Hackbusch2012}.
 They consist in approximating the solution in a subset $\curM\subset X_{N}$ of low-rank tensors. Different tensor structures of space $X_{N}$ can be exploited, such as  $X_{N}=V_{q}\otimes W_{p}$ or $X_{N}=V_{q}\otimes W_{p}^{\bfy}\otimes W_{p}^{\bfz}$. Also, if the measures $\Gamma_{N_{g}}$ and $\Gamma_{\nu}$ are tensor product of measures, then $W_{p}^{\bfy}$ and $W_{p}^{\bfz}$ can be chosen as   tensor product approximation spaces  $W_{p}^{\bfy} = \otimes_{k=1}^{N_{g}} W_{p}^{\bfy,(k)}$ and $W_{p}^{\bfz} = \otimes_{k=1}^{\nu} W_{p}^{\bfz,(k)}$.  Low-rank approximation methods
then consist in approximating the solution in a subset $\curM\subset X_{N}$ of low-rank tensors, which is a subset of low dimension in the sense that $\curM$ can be parameterized by a small number of parameters (small compared to the dimension of $X_{N}$). Different low-rank formats can be used, such as canonical format, Tucker format, Hierarchical Tucker format or more general tree-based tensor formats (see e.g. \cite{Hackbusch2012}). These tensor subsets can be formally written as
$$
\curM = \left\{ v = F_{\curM}(w_{1},\hdots,w_{\ell}) ; w_{1} \in \RR^{r_{1}},\hdots,w_{\ell} \in \RR^{r_{\ell}} \right\}
$$
where $F_{\curM}$ is a multilinear map with values in $X_{N}$ and where the $w_{k}\in\RR^{r_{k}}$ ($k=1,\hdots, \ell$) are the parameters. The dimension of such a parametrization is  $\sum_{k=1}^{\ell} r_{k}$. As an example, $\curM$ can be chosen as the set of rank-$m$ canonical tensors in $X_{N}= V_{q}\otimes W_{p}^{\bfy}\otimes W_{p}^{\bfz}$, defined by
$\curM = \{ v = \sum_{i=1}^{m} w_{i}^{\bfx} \otimes w_{i}^{\bfy} \otimes w_i^{\bfz}; w_{i}^{\bfx}\in V_{q}, w_{i}^{\bfy}\in W_{p}^{\bfy},w_{i}^{\bfz}\in W_{p}^{\bfz}\}.$

\paragraph{Algorithms for the low-rank approximation of $u_{N}\in X_{N}$.}

The Galerkin approximation $u_{N}\in X_{N}$ of $u$ is the unique minimizer of the strongly convex functional $J:X_{N}\rightarrow \RR$ defined by  $J(v) = \frac{1}{2} a(v,v) - F(v)$. A low-rank approximation $u_{r} \in \curM_{r}$ of $u_{N}$ can then be obtained by solving the optimization problem
$$
\min_{v\in\curM} J(v) = \min_{w_{1}\in \RR^{r_{1}},\hdots,w_{\ell}\in \RR^{r_{\ell}}} J(F_{\curM}(w_{1},\hdots,w_{\ell})).
$$
This problem can be solved using alternated minimization algorithms or other optimization algorithms (see \cite{Hackbusch2012}).  Also, greedy procedures using low-rank tensor subsets can be introduced in order to construct a sequence of approximations $\{u^{k}\}_{k}$ with $u^{k+1} = u^{k} + w^{k+1}$ and $w^{k+1}\in \curM$ defined by
$
J(u^{k}+w^{k+1}) = \min_{v\in\curM } J(u^{k} + v). 
$
We refer to \cite{Falco2012} for the analysis of a larger class of greedy algorithms in the context of convex optimization problems in tensor spaces, and to \cite{Nouy2010} for the practical implementation of some algorithms in the context of stochastic parametric partial differential equations.

{\begin{rem}[On the approximability of the solution using  low-rank approximations]
A very few results are available concerning the complexity reduction that can be achieved using  low-rank approximations, in particular for the approximation of high-dimensional maps arising in the context of high-dimensional parametric stochastic equations. 
Of course, results about sparse polynomial approximation methods (e.g. \cite{Cohen2013}) could be directly translated in the context of low-rank tensor methods (a sparse decomposition on a multidimensional polynomial basis being a low-rank decomposition) but this would not allow to quantify the additional reduction that could be achieved by low-rank methods. 
Providing a priori convergence results for low-rank approximations for the proposed class of random fields would deserve an independent analysis which is out of the scope of the present paper. However, we notice that some standard arguments could be used in order to provide a priori convergence results of low-rank approximations in the 2-order tensor space $X = H_0^1(D) \otimes L^2_{\Gamma}(\mathbb{R}^\mu) = L^2_{\Gamma}(\mathbb{R}^\mu ; H_0^1(D))$. Indeed, under suitable regularity conditions on the samples of the random fields in the class (related to properties of the covariance structure of the algebraic prior), the set of solutions $\mathcal{U} = \{u(\cdot,\bfy,\bfz);(\bfy,\bfz)\in \mathbb{R}^{\mu}\}$ could be proven to be a subset of $H^1_0(D)$ with rapidly convergent Kolmogorov $n$-width\footnote{{For example, if the set of solutions $\mathcal{U} $ is a bounded subset of $H^2(D)$ (which can be obtained with a simple a posteriori analysis if the random fields have samples with $C^1$ regularity), then the Kolmogorov  $n$-width can be proven to converge as $n^{-1/d}$ where $d$ is the spatial dimension.}}, therefore proving the existence of a good sequence of low-dimensional spaces in $H^1_0(D)$ for the approximation of the set of solutions $\mathcal{U}$. At least the same convergence rate (with respect to the rank) could therefore be expected for low-rank approximations that are optimal with respect to a $L^2$-norm in the parameter domain. 
These questions will be investigated in a future work.
\end{rem}
}

 \subsection{Remarks about the identification procedure}

The identification procedure has been presented in Section~\ref{Section3.3}. In view of the numerical analysis presented above, the following remarks can be done.
Different approximations of the map $u$ should be constructed at the different
steps of the identification procedure. Indeed, the quality of
approximations of $u$ clearly depends on the choice of measure $\Gamma_{\nu}$ on the parameters space $\RR^{\nu}$. It is recalled that
the parametrization of the class  $\curK=\{\curK^{(m,N)}(\cdot,\bfXi,\curM_{[a]}(\bfz)) \, ;\, \bfz\in \RR^{\nu} \}$ of random fields depends on the set of parameters $[a] = \curM_{[a]}(0)$. The map $u$ depends on the choice of $[a]$ but it is clearly independent on the choice of measure $\Gamma_{\nu}$, provided that the support $supp(\Gamma_{\nu})=\RR^{\nu}$. However, the Galerkin approximation $u_{N}$ being optimal with respect to a norm that depends on the measure $\Gamma_{\nu}$, the quality of the approximation clearly depends on measure $\Gamma_{\nu}$.
For that reason, the choice of measure $\Gamma_{\nu}$ can be updated throughout the
identification procedure in order to improve the approximation of
$u$ as explained at the end of Section~\ref{Section3.3}. Step~6 and step~7 described in Section~\ref{Section3.3} can then be clarified as follows.

At \emph{Step 6}, the parametrization of $\curK$ around $[a]=[y_{0}]$ is used and
 an approximation of the corresponding map
$u$ in $X$ is constructed using low-rank tensor methods. The use of this
explicit map allows a fast maximization of the likelihood function to be performed,
yielding an estimation $\check{\mathbf{z}}$ of $\bfz$  (or equivalently  $[\check{y}]$ of $[y]$).

At \emph{Step 7}, the
map $u$ constructed in Step 6 could be used for the subsequent Bayesian update.
However, it is also possible to construct a new map that takes into
account the result of the maximum likelihood estimation. Therefore, the parametrization of $\curK$ around
$[a]={[\check{y}]}$ is used and the boundary value problem is solved again in order to obtain
an approximation of the corresponding map $u(\cdot,\cdot,\cdot;{[\check{y}]})$ in $X$.
This map can
then be efficiently used for solving the Bayesian update problem.

%
%
%

\section{Conclusion}
\label{Section5}
In this paper, we have presented new results allowing an unknown non-Gaussian positive matrix-valued random field to be identified through a stochastic elliptic boundary value problem, solving a statistical inverse problem. In order to propose a constructive and efficient methodology and analysis of this challenging problem in high stochastic dimension, a new general class of non-Gaussian positive-definite matrix-valued random fields, adapted to the statistical inverse problems in high stochastic dimension for their experimental identification, has been introduced. For this class of random fields, two types of representation are proposed: the exponential type representation and the square type representation. Their properties have been analyzed. A  parametrization of discretized random fields belonging to this general class has been proposed and analyzed for the two types of representation. Such parametrization has been  constructed using a polynomial chaos expansion with random coefficients and a minimal parametrization of the compact Stiefel manifold related to these random coefficients. Using this parametrization of the general class,  a complete identification procedure has been proposed. Such a statistical inverse problem requires to solve the stochastic boundary value problem in high stochastic dimension with efficient and accurate algorithms. New results of the mathematical and numerical analyzes of the parameterized stochastic elliptic boundary value problem have been presented. The numerical solution provides an explicit approximation of the application that maps the parameterized general class of random fields to the corresponding set of random solutions. Since the proposed general class of random fields possibly contain random fields which are not uniformly bounded, a particular mathematical analysis has been developed and  dedicated approximation methods have been introduced. In order to obtain an efficient algorithm for constructing the approximation of this very high-dimensional map, we have {described possible} complexity reduction methods using {sparse or} low-rank approximation methods that exploit the tensor structure of the solution which results from the parametrization of the general class of random fields.
{
In this paper, we do not provide any result concerning the complexity reduction that can be achieved when using low-rank or sparse approximation methods for the proposed class of random fields. These results would require a fine analysis of the regularity and structures of the solution map. Some recent results are already available for sparse approximation methods and for some simpler classes of random fields (see \cite{Cohen2013}). This type of results should be extended to the present class of random fields. Concerning low-rank methods, a very few quantitative convergence results are available. Some recent results are provided in \cite{Schneider2013} for the approximation of functions with Sobolev-type regularity. For the present class of random fields, specific analyses are necessary to  understand the structure of the solution map induced by the proposed parametrization and to quantify the complexity reduction that could be achieved with low-rank methods. These challenging issues will be addressed in future works. 
} 

\section*{Acknowledgements}
This research was supported by the "Agence Nationale de la Recherche", Contract TYCHE, ANR-2010-BLAN-0904.

\label{lastpage}

\end{document}